\documentclass[11pt]{amsart}
\usepackage{youngtab,amsmath,enumerate}
\usepackage{tikz-cd}

\newcommand{\cL}{\mathcal{L}}
\newcommand{\ce}{\mathcal{E}}
\newcommand{\cc}{\boldsymbol{c}}

\newcommand{\W}{W}

\newcommand{\cT}{\mathcal{T}}

\newcommand{\bg}{\boldsymbol{g}}
\newcommand{\g}{\boldsymbol{g}}
\newcommand{\si}{\sigma}

\newcommand{\ro}{\mathsf{P}}
\newcommand{\sJ}{\mathsf{J}}
\newcommand{\rr}{\mathbb{R}}
\newcommand{\D}{\mathbb{D}}
\renewcommand{\d}{\mathrm{d}}
\newcommand{\hD}{\mathbb{D}}
\newcommand{\K}{\mathbb{K}}
\newcommand{\J}{\mathsf{J}}

\usepackage{amsmath,amssymb,amsthm}

\usepackage{cancel}

\def\sideremark#1{\ifvmode\leavevmode\fi\vadjust{\vbox to0pt{\vss% the remark
 \hbox to 0pt{\hskip\hsize\hskip1em%                          will appear only
 \vbox{\hsize2cm\tiny\raggedright\pretolerance10000%          on the side
  \noindent #1\hfill}\hss}\vbox to8pt{\vfil}\vss}}}%
\theoremstyle{definition}
\newtheorem{definition}{Definition}[section]
\newtheorem{remark}[definition]{Remark}

\theoremstyle{plain}
\newtheorem{lemma}[definition]{Lemma}
\newtheorem*{lem*}{Lemma}
\newtheorem{proposition}[definition]{Proposition}
\newtheorem{corollary}[definition]{Corollary}
\newtheorem{theorem}[definition]{Theorem}

\numberwithin{equation}{section}

\newcommand{\EE}{\mathcal{E}}
\newcommand{\E}{\mathcal{E}}

\title{Conformal Killing tensors and their Killing scales}

\begin{document}

\author{A.~Rod Gover}
\address[Gover]{Department of Mathematics, The University of Auckland, Private Bag 92019, Auckland 1142, New Zealand}
\email{r.gover@auckland.ac.nz}
\author{Jonathan Kress}
\address[Kress]{School of Mathematics and Statistics, University of New South Wales, Sydney 2052, Australia}
\email{j.kress@unsw.edu.au}
\author[Thomas Leistner]{Thomas Leistner}\address[Leistner]{School of Computer and Mathematical Sciences, University of Adelaide, SA~5005, Australia}\email{thomas.leistner@adelaide.edu.au}

\subjclass[2010]{Primary 
53A30;
%conformal structures on manifolds
 Secondary 53C18, 53A20, 53B15, 70H33, 37K10}

\thanks{This work was supported by 
the Australian Research
Council via the grant DP190102360 and
by the 
 Royal Society of New Zealand via Marsden Grant
 19-UOA-008. The authors would like to thank  Australia's international and residential mathematical research institute MATRIX, where some  work on the paper took place.
 }

\begin{abstract}
We address the problem of how to characterise when a rank-two
conformal Killing tensor is the trace-free part of a Killing tensor
for a metric in the conformal class. We call such a metric a Killing
scale.  Our approach is via differential prolongation using
conformally invariant tractor calculus. First, we show that there is a
useful partial prolongation of the conformal Killing equation to a
simplified equation for sections of some tractor bundle. We then use
this partial prolongation to provide such an invariant
characterisation in terms of the scale tractor and this partial
prolongation.  This captures invariantly the relevant
Bertrand--Darboux equation. We show that Einstein Killing scales have a
special place in the theory. On conformally flat manifolds, we give
the full prolongation of the conformally Killing equation to a
conformally invariant connection on a tractor bundle. Using this, we 
provide a characterisation of (non-scalar flat) Einstein Killing
scales by an algebraic equation for the scale tractors corresponding
to such metrics. This also provides an algebraic description of the
linear subspace of conformal Killing tensors that are compatible with
a given Einstein Killing scale.  For completeness and to introduce the
main ideas, we also study analogous questions for conformal Killing
vectors.  
\end{abstract}

\maketitle

\setcounter{tocdepth}{1}
%\tableofcontents

\section{Introduction}

Recall that on a pseudo-Riemannian manifold $(M,g)$ a vector field $k$
is called a {\em Killing vector field} if $\cL_k g=0$, where $\cL_k$ denotes the Lie derivative along the flow of $k$. Using abstract index notation and dualising the vector field to a $1$-form using the metric, this can be written
as $\nabla_{(a}k_{b)}=0$ in terms of the Levi-Civita connection
$\nabla$. Here and throughout $( a \ldots b)$ indicates
the symmetric part over the enclosed indices.  More generally a  symmetric tensor
$k_{a_1 \ldots a_\ell}\in \Gamma(S^\ell T^*M) $ is called a {\em Killing tensor} if it
satisfies the equation   \begin{equation}\label{Kt}
  \nabla_{(a_0} k_{a_1\cdots a_\ell)}=0.
\end{equation}
Although not directly providing symmetries or infinitesimal symmetries
of the structure $(M,g)$ when $k\geq 2$, such tensors are important
for a number of reasons. Firstly they provide first integrals for
geodesics: If a curve $\gamma$ is a geodesic for $(M,g)$, meaning $
\nabla_{u}u =0 $ where $u=\dot\gamma$ denotes the curve velocity, then
is is easily verified that
$$
u^{a_1}\cdots u^{a_\ell} k_{a_1\cdots a_\ell}
$$ is necessarily constant along the curve, cf.~\cite{Sommers73}. This
and its variants have important applications in many fields,
\cite{Carter68,AnderssonBlue15}.  Killing tensors also arise as the
leading symbol of linear differential operators that commute with the
Laplacian (and so preserve its spectrum)
\cite{eastwood05,MichelRadouxSilhan14,MichelSombergv-Silhan17}, or
similarly that commute with the Hamiltonian in many physical systems
--- where they are part of the story of so-called hidden symmetries
\cite{Cariglia14}.  This has a long history related to Maupertuis
principle and, for example, Jacobi Geometrisation
\cite{Eisenhart28,Benn06}, which provide (in suitably limited
circumstances) a way of encoding dynamics geometrically by passing to
a conformally related metric in which the trajectories are geodesics.
The St\"{a}ckel Transform is another incarnation of this
\cite{BoyerKalninsMiller86}.  Systems of Killing tensors are
linked to separation of variables,
\cite{Miller77book,KalninsMiller82,KalninsKressMiller18}, and special
functions arise in relating the different separating coordinate
systems, such as for the Askey-Wilson Scheme.  For several decades now
there has been an interest in systems that are superintegrable ---
this means a pseudo-Riemannian (or complex) manifold that admits a
large number of Killing tensors (see \cite{KressSchobelVollmer24} for
references).  Most commonly the assumption is the existence of a
system of $2n-1$ functionally independent rank-2 Killing tensors
(where the metric is counted as one). On such spaces one can typically
algebraically compute the trajectories of solutions to the
Hamilton-Jacobi equations.

The conformal Killing equation on  trace-free (with respect to the metric) rank $k$ tensors is 
\begin{equation}\label{cKt}
\operatorname{Trace-free-part}(\nabla_{(a_0} k_{a_1\cdots a_\ell)})=0,
\end{equation}
where ``$\operatorname{Trace-free-part}$'' is the projector onto the
part that is trace-free.  In a suitable sense this equation is
conformally invariant (see the end of Section~\ref{density-sec}), and
is a weakening of \eqref{Kt} in that (obviously), given a solution of
\eqref{Kt}, its trace-free part solves \eqref{cKt}. For the case of
Killing vectors (i.e. $\ell=1$) the converse to this last statement
has been of considerable historical interest, especially in relation
to the global issues of whether or not a vector field is {\em
  essential} --- meaning that there is no conformally related metric
for which it is Killing, see for example \cite{Frances07},
\cite{KuhnelRademacher08}, and references therein. If the length of a
conformal vector field has no zeros, then there is always at least one
Killing scale by a simple argument that we review in Section
\ref{vfKscale}.

For higher rank tensors the situation is more subtle. The point is rank one
tensors (vector or 1-form fields) have no trace part. However for
$\ell\geq 2$ there are two natural questions. 
\medskip

\noindent {\em Question 1:} {Given a solution $k_{a \ldots c}$ to \eqref{cKt} on
  $(M,g)$, is there a conformally related metric $\widehat{g}_{ab}=\Omega^2
  g_{ab}$ }  for which $\widehat{k}_{a \ldots c}=\Omega^{2\ell} k_{a \ldots c}$ solves \eqref{Kt}?

\medskip

\noindent If there is such an $\Omega$ then we term $\widehat{g}$ a {\em strong
  Killing scale}. 

\medskip

\noindent {\em Question 2:} {Given a solution $k_{a \ldots c}$ to \eqref{cKt} on
  $(M,g)$ is there a conformally related metric $\widehat{g}_{ab}=\Omega^2
  g$ } and $ \widehat{k}_{a \ldots c}\in \Gamma(S^\ell T^*M) $  that solves \eqref{Kt} on $(M,\widehat{g})$ and satisfies
$$
k_{a \ldots c}=  \Omega^{-2\ell} \operatorname{Trace-free-part} (\widehat{k}_{a \ldots c})?
$$

\smallskip

\noindent If there is such an $\Omega$ then we term $\widehat{g}_{ab}$ a {\em Killing scale}. 
Note that, in both cases, the power of 
 $-2\ell$ for $\Omega$ is forced, so this is no further restriction.

The condition of the second question is strictly weaker, but is more
difficult to treat because both $\Omega$ and the trace adjustment are
unknowns. It is the more important question for applications as it
leads to a broader class of cases where solutions to
\eqref{Kt} for one metric may also yield solutions (after scaling and
trace-adjustment) to \eqref{Kt} for conformally related metrics.

There are variants of these questions where one fixes the metric
$g_{ab}$ and asks for the subspace of solutions $k_{a \ldots c}\in
\Gamma(S^\ell_0T^*M)$ (where $S^\ell_0$ indicates the trace-free
symmetric tensor power part) to \eqref{cKt} that solve \eqref{Kt}, or,
respectively, after the addition of a trace part solve \eqref{Kt}.
This perspective is important for (for example) the study
superintegrable systems, where one seeks a number of Killing tensors
that simultaneously solve \eqref{Kt} in a given scale $g_{ab}$.

In the following we focus on the case of rank $\ell=2$, as it is the
simplest $\ell>1$ case, and also because, at least for current theory, it
is main case for the applications discussed earlier. Nevertheless to
establish our approach and set up the machinery in a simpler context we
first treat the rank-1 case of vectors.

Section \ref{rank1-sec} contains the study of conformal Killing
vectors and Killing vectors.  We will work on a conformal class of
metrics $\cc$, meaning $\cc$ is an equivalence class of metrics (of
some signature) where $g,g'\in \cc$ means that $g'=\Omega^2 g$ for
some positive smooth function $\Omega$.  Proposition \ref{Kscalevec}
establishes that a conformal Killing vector $k_a$ is Killing for a
metric in $g\in\cc$ if and only if there is an orthogonality between a
partial tractor prolongation of $k$ (that we denote $K_A$) and the
{\em scale tractor} $I_A$ that corresponds to the metric.  Similar
notions of orthogonality to the scale tractor play an important role
throughout the article. As mentioned above, it well known that away
from zeros of $g_{ab}k^ak^b$ there is always at least one such scale
--- and for completeness we review this result in Corollary
\ref{aKscale}.

From the tractor perspective it is immediately clear that Einstein
scales (meaning metrics in $\cc$ such that the Ricci tensor is pure
trace) must play a special role if they exist --- as they each
correspond to a scale tractor that is parallel for the tractor
connection \cite{bailey-eastwood-gover94, gover-nurowski04}.
Proposition~\ref{newkillingvec} shows that, in an Einstein scale, a
non-trivial conformal Killing vector always determines a Killing
vector. Typically this is non-trivial.  Although we have not found it
in the literature, it seems likely that this result for conformal
vectors was known previously. However the approach here shows that
similar results are immediately available in higher rank, and we in
particular explore this conformal Killing tensors of rank~$2$, see
Proposition~\ref{newkillingtensor}.

While Proposition \ref{newkillingtensor} is independently interesting
and  potentially useful in the programmes of applications mentioned
earlier, the questions mentioned above form the main focus of Section~\ref{rank2-sec}. The key idea is that, since
the Killing tensor equation \eqref{Kt} and the conformal Killing
tensor equation \eqref{cKt} are both overdetermined finite type PDEs,
it is natural to study them via differential prolongation 
\cite{BryantChernGardnerGoldschmidtGriffiths91,BransonCapEastwoodGover06}.

When one prolongs such an equation, the aim is replace the original
equation with a larger but closed (and hence simpler) first order
system. Doing this na\"{\i}vely involves may choices of variables and
steps in linear algebra (or Lie alegebra cohomology). However as these
equations fall into the class of so-called first BGG equations
\cite{CapSlovakSoucek01,CalderbankDiemer01} one can instead efficiently use the tractor calculus
associated to the Cartan connection. See \cite{GoverLeistner19} for a prolongation
of the Killing tensor using projective tractor calculus. This not only
achieves the prolongation, it packs the data into well understood
tensor products of low rank natural bundles, and by elementary
representation theory (or from \cite{CapSlovakSoucek01,CalderbankDiemer01}) one knows at the outset
what these bundles will be. In fact in the case of projectively or
conformally flat backgrounds, the prolonged system {\em is} just an
application of the usual tractor connection.

Thus, in Section \ref{background-sec} we introduce the conformal and
projective tractor calculus. For example, in the case of manifolds of
dimension $n$ equipped with just a  conformal structure $\cc$ (but no
preferred choice of metric therein) there is, in general, no
distinguished connection on the tangent bundle, but there is on a
bundle $\cT$ (which we will also denote by $\ce^A$ in abstract index notation) of rank $(n+2)$ that
extends this.  This canonical conformally invariant linear connection
$\nabla^\cT$ on $\cT$ is the conformal tractor connection. The earlier
mentioned scale tractor $I^A$ is a section of $\cT$, and if
parallel for $\nabla^\cT$ determines (and is equivalent to), on an open
dense set, a metric $g\in \cc$ that is Einstein. Here and throughout
we assume that $M$ is connected. The connection $\nabla^\cT$ of course
extends to tensor powers $\E_{A\ldots C}$ of $\cT$ and in fact preserves a bundle
metric ---  the tractor metric --- that can be used to raise and
lower tractor indices. There are other key objects: for example a
canonical tractor $X^A$ that governs the filtration of $\cT$, and a
Thomas-D operator $\mathbb{D}_A$ that acts between (density weighted)
tractor bundles and behaves like a covariant derivative on tractors but 
is of second order.

An interesting feature of using the tractor machinery to do
prolongation is that it reveals an important partial prolongation (see
also \cite{GoverLeistner19}). The following theorem summarises the
results of Theorems~\ref{theo0a} and~\ref{theo0b}, which capture this
half-prolongation and the rank 2 analogue of
Proposition~\ref{Kscalevec}. This uses the notation and definitions of
Section \ref{background-sec}.

\begin{theorem}\label{theo0}
On a conformal manifold there is a one-to-one correspodence between conformal Killing tensors of rank $2$ and 
symmetric trace-free tractors $K_{(AB)_0}\in \E_{(AB)_0}[2]$ such that
$X^AK_{AB}=0$ and 
\[\D_{(A}K_{BC)}=X_{(A}F_{BC)}\] for some section $F_{AB}$ of $\E_{(AB)_0}[2]$ with $X^AF_{AB}=0$.

Moreover, given a conformal Killing tensor $k_{ab}$ with corresponding $K_{AB}$, there is a one-to-one correspondence between 
metrics $g_{ab}$ in the conformal class for which   $k_{ab}$ is the tracefree part of a Killing tensor for the metric $g_{ab}$
and positive  scale tractors $I_A=\D_A\sigma$ 
 such that
\[
 I^AK_{AB}= \tfrac{\sigma}{2} \D_B\lambda - \lambda I_B+X_B F\]
 for some sections $F$ of  $\mathcal E[1]$ and   $\lambda$ of $\E[2]$. 
For the metric $g_{ab}=\sigma^{-2}\g_{ab}$, the Killing tensor is 
$\bar{k}_{ab}=k_{ab}+\lambda g_{ab}$.
\end{theorem}

Although this result appears more technical than Proposition~\ref{Kscalevec} for Killing vectors, it is a key building block for later developments.
The map from conformal Killing tensors to symmetric rank-$2$ tractors
is given as 
\begin{equation}\label{rk2half-intro}%%$$
k_{ab}\longmapsto \begin{pmatrix}
  0 & 0 & 0\\
  0 &k_{ab}  & \rho_b\\
  0 & \rho_a & \rho
\end{pmatrix} \stackrel{g}{=} K_{AB}
 \end{equation}%% $$
where
\begin{equation}
\label{rhodef} \rho_a:= \tfrac{2}{n+2} \nabla^bk_{ab} \quad\text{and}\quad \rho:= \tfrac{1}{n+1} 
\left( \tfrac{1}{2}
 \nabla^a\rho_a +P^{ab}k_{ab}\right). 
 \end{equation}

Next, we restrict to conformally flat manifolds and derive equivalent
tractorial conditions for the tensor $k_{ab}$ to be conformally
Killing. The proofs of the following equivalences can be found in
Section~\ref{theo1-sec}.
\begin{theorem}\label{theo1}
On a conformally flat semi-Riemannian manifold, there is a one-to-one correspondence between
\begin{enumerate}[(i)]
\item \label{theo1-1} rank $2$ conformal Killing tensors $k_{ab}$,
\item \label{theo1-2} tracefree symmetric tractors $K_{BC}$ with $X^AK_{AB}=0$ and $\D_{(A}K_{BC)}=0$,
\item \label{theo1-3} parallel conformal tractors $\W_{ABCD}$ with Weyl tensor symmetries.
\end{enumerate}
\end{theorem}

The map from rank $2$ tractors to rank $4$ tractors with Weyl
symmetries, denoted by $\mathcal W$, is the composition of
\begin{equation}\label{K2toK4} K_{AB}\longmapsto \K_{ABCD}=\D_A\D_CK_{BD},\end{equation} 
and 
\begin{equation}\label{K4toW4}\K_{ABCD}\longmapsto \W_{ABCD}=\K_{[AB][CD]}.
\end{equation}
The equation that $W_{BCDE}$ is parallel,
$
\nabla^\cT_a W_{BCDE} =0
$,
is the full prolonged system for the conformal Killing equation in
the conformally flat setting. Any solution of \eqref{cKt} is
determined by the corresponding $W_{BCED}$ at one point. Conversely,
on simply connected manifolds, \eqref{cKt} is solved globally by
simply specifying such a tractor at a point and extending it to the
manifold by (tractor) parallel transport.  Hence, in this setting, the parallel tractors with Weyl symmetries in~(\ref{theo1-3}) of Theorem~\ref{theo1} are isomorphic to the vector space  
$\mathfrak{W}\subset \otimes^4(\mathbb{R}^{n+2})^*$ of covariant rank $4$ tensors  of $\mathbb{R}^{n+2}$  with Weyl symmetries (i.e. trace-free curvature tensors),
\[W_{ABCD}=-W_{BACD}=-W_{ABDC},\quad W_{[ABC]D}=0, \quad W_{ABC}{}^B=0.\]
This last equivalence in Theorem~\ref{theo1} vastly
simplifies the understanding of the equation \eqref{cKt}. This  last part
can  easily be deduced by the BGG machinery of
\cite{CapSlovakSoucek01,CalderbankDiemer01} (see also 
\cite{cgh11,GoverSnellTaghavi}
 where we summarised
explicitly how to apply it in settings such as this). But that machinery
does not proceed through the conformally invariant partial steps that
we need and produce in our treatment here. Indeed the half-prolonged object $K_{CE}$ is simply
$$
X^BX^DW_{BCDE},
$$
up to a non-zero constant factor.

Finally, still in the conformal flat setting, in
Section~\ref{theo2-sec} we apply Theorem~\ref{theo1} to obtain an
extremely simple characterisation of Killing scales that are Einstein.
\begin{theorem}\label{theo2}
On a conformally flat semi-Riemannian manifold, let $k_{ab}$ be a conformal Killing tensor, and let $K_{AB}$
and $\W_{ABCD}$ be the corresponding tractors from Theorem~\ref{theo1}. Then there is a one-to-one correspondence between
\begin{enumerate}[(i)]
\item \label{theo2-1} scales $\sigma$ such that the metric $g_{ab}=\sigma^{-2}\g_{ab}$ is Einstein, and $k_{ab}$ is the tracefree part of a Killing tensor for $g_{ab}$,
\item \label{theo2-2}   parallel scale tractors $I_C$ with $\iota:=I_CI^C\not=0$ that satisfy
\begin{equation}\label{K2eq} I^AK_{AC}=\kappa I_B- \tfrac{\sigma}{2} \D_B\kappa ,
\end{equation}
 \item  \label{theo2-3} parallel scale tractors $I_C$ with $\iota=I_CI^C\not=0$ that satisfy
$I^A\W_{AB[CD}I_{E]}=0$.
\end{enumerate}
\end{theorem}
The equivalence between (\ref{theo2-1}) and  (\ref{theo2-3}) is established in Theorem~\ref{theo2a} and the equivalence between (\ref{theo2-1}) and  (\ref{theo2-2}) is the statement of Corollary~\ref{theo2b}.

These results are formulated for a fixed conformal Killing tensor, but
they can be viewed from a different perspective. When fixing the
(Einstein) scale $\sigma$ in the above results, they characterise the
vector space $\mathfrak{C}(\sigma)$ of conformal Killing tensors that
can be simultaneously trace adjusted to Killing tensors in the scale
$\sigma$ as the kernel of the linear map
$$
\mathfrak{W}\ni \W_{ABCD} \longmapsto
\label{IWwedgeIintro} I^A\W_{AB[CD}I_{E]},
$$
 We emphazise that this vector space is
finite dimensional and it is not hard to see that it is isomorphic to a codimension $1$ subspace in the space of Riemannian curvature tensors of 
$\mathbb R^{n+1}$, see Corollary~\ref{dim-cor}. We arrive at the following conclusion.
\begin{corollary}\label{dim-cor-intro}
On a conformally flat manifold of dimension $n$, let
$\si$ be a non-Ricci-flat Einstein scale. Then the 
space of rank $2$ conformal Killing tensors for which $\sigma$ is a Killing scale is of dimension $\frac{n(n+1)^2(n+2)}{12}-1$.
\end{corollary}

The restriction to conformally flat geometries is of course
restrictive. The commonly used requirement of $2n-1$ conformal Killing
tensors would imply this in low dimensions. The last results concern
Einstein scales, which is again a restriction. However our point here
is that the role of Einstein scales are distinguished --- and that
feature has emerged here in a very clear way. Secondly it is clear
from our understanding of prolongations (see
e.g.~\cite{HammerlSombergSoucekSilhan12}) that more general results
that include non-Einstein scales and also conformally curved
geometries have to be consistent with and include the picture here. So
we regard these results as key first steps, and a road map, for
understanding the general picture.

\section{Conformal and projective tractor bundles}
\label{background-sec}

In this section we introduce the main tools and notation that we need
for the later developments. In particular we introduce the conformal
and projective tractor bundles and some of some of the associated tractor
calculus. Where possible we have made the presentation self contained,
and mainly follow the sources
\cite{bailey-eastwood-gover94,curry-Gover,Gover-Peterson-Lap} where further details may
also be found.

\subsection{Notation and conventions} 
\label{conventions}
Let $M$ be a smooth manifold of dimension $n$.
 In the following we write 
\[\EE^a=TM, \quad \EE_a=T^*M, \quad\EE_{(ab)}=\odot^2T^*M, \quad \EE_{[ab]}=\Lambda^2T^*M,\] 
etc.~for the tensor bundles of $M$. In the presence of a
semi-Riemannian metric $g_{ab}$ we denote by $ \EE_{(ab)_0}$ the
tensors $h_{ab}\in \EE_{(ab)}$ that are trace-free with respect to the
metric $g_{ab}$, i.e.~for which $h_a{}^a=h_{ab}g^{ba}=0$.  We denote
$\nabla_a$ the Levi-Civita connection and 
\[R_{ab}{}^c{}_dX^d=\nabla_a\nabla_b X^c -\nabla_b\nabla_a X^c, \]
defines the Riemannian curvature tensor of $g_{ab}$, acting on a
tangent vector field $X^a$. Its trace $R_{bd}=R_{ab}{}^a{}_d$ is the
Ricci tensor and $R=R_a{}^a$ the scalar curvature.  Moreover,
$\ro_{ab}$ denotes the conformal Schouten tensor, and is defined by
\[ R_{ab}=(n-2) \ro_{ab}+\J g_{ab},\quad\text{ where }
\J=\ro_a{}^a =\tfrac{1}{2(n-1)} R, \]
while
\[W_{abcd}=R_{abcd}-2( g_{c[a}\ro_{b]d}+ g_{d[b}\ro_{a]c})\]
is the conformal Weyl tensor, which is completely trace free. By
square brackets we denote the skew symmetrisation over enclosed
indices, with the convention that $\omega_{a_1 \ldots
  a_\ell}=\omega_{[a_1 \ldots a_\ell]}$ for an $\ell$-form
$\omega_{a_1 \ldots a_\ell}$. Similarly, round brackets denote the
symmetrisation.

\subsection{Conformal densities} 
\label{density-sec}
Let $(M, \cc)$ be a conformal manifold of signature $(r,s)$. 
The
conformal class $\cc$ is equivalent to a reduction of the frame bundle
$\mathcal{GL}(M)$ of $M$ to the bundle  $\mathcal{CO}(M,\cc)$ of frames that 
are orthonormal for a $g\in \cc$, which has the structure group
$G_0=\mathrm{CO}(r,s)$.  For $w\in \rr$, let $\delta^w$ be the
representation of $G_0=\mathrm{CO}(r,s)$ on $\rr$ defined by
$\delta^w(A)=|\det(A)|^{w/n}$.  
Using this we can define density
bundles as line bundles associated to this representation,
\[\EE[w]=\mathcal{CO}(M,\cc)\times_{\delta^w}\rr.\]
The sections of $\EE[w]$ are identified with smooth maps
$s: \mathcal{CO}(M,\cc)\to \rr$ such that \[ s(f\cdot A)=|\det(A)|^{ -w/n} s(f)\quad\text{ for all }\quad f\in \mathcal{CO}(M,\cc).\]
In particular, $\EE[-2n]\cong \otimes^2(\Lambda^n T^*M)$. For each $w\in \mathbb{R}$ the density bundle $\ce [w]$ is oriented and we write $\ce_{\pm}[w]$ for (respectively) the positive and negative ray subbundles.
We can use a density $\sigma\in \ce_{\pm}[1]$ to trivialise  all density bundles: if $\tau \in \EE[w]$, then 
\[\sigma^{-w}\tau\in C^\infty (M)\]
is a function. 

If $\rho$ is the standard representation of $G_0$ on $\mathbb{R}^n$ and $\rho^*$ its dual, then the (co-)tangent bundle are associated bundles 
\[\EE^a=\mathcal{CO}(M,c)\times_\rho\mathbb{R}^n, \qquad  \EE_a=\mathcal{CO}(M,c)\times_{\rho^*}(\mathbb{R}^n)^*.\]
Hence we can form tensor products $\EE_a[w]=\EE_a\otimes \EE[w]$, etc., to get weighted tensor bundles. 
The conformal class $\cc$ is equivalent to the
{\em conformal metric} $\mathbf{g}_{ab}$, which is a canonical section
of $\EE_{(ab)}[2]$ that is given by 
\[\mathbf{g}_{ab}|_p:=[ f,\eta_{ab}]\otimes [ f, 1],
\]
where $f\in \mathcal{CO}(M,\cc)|_p$ is a conformal frame and
$\eta_{ab}$ represents the standard pseudo-Euclidean scalar product.
Using $\mathbf{g}_{ab}$ we can invariantly raise and lower the indices of
weighted tensors. For each $g_{ab}\in \cc$, the Levi-Civita connection
$\nabla$ yields a connection on the bundle
$\mathcal{CO}(M,\cc)$ and hence on each of the
conformal density bundles. This connection is given by the formula
\[\nabla\tau =\sigma^w \, \mathrm{d}(\sigma^{-w}\tau),\]
for $\tau \in\EE[w]$, $\bg_{ab}=\sigma^2g_{ab}$ for $\sigma \in \EE_+[1]$, and $\nabla$ the Levi-Civita connection of $g_{ab}$.
It implies that $\nabla\sigma=0$ and that that each connection $\nabla$ on $\EE[w]$ is flat. It also shows that $\nabla_a\bg_{bc}=0$.

The connections induced on the density bundles $\EE[w]$, $\EE^a[w]$ and $\EE_a[w]$ from the Levi-Civita connections $\nabla$ and $\hat\nabla $ of $g$ and $\hat{g}=\Omega^2g$ from the conformal class are related as follows
\begin{equation}\label{trafo}
\begin{array}{rcl}
\hat\nabla_a\sigma &=& \nabla_a\sigma+w\,\sigma\, \Upsilon_a,\\
\hat\nabla_aV^b &=& \nabla_aV^b +(w+1)\Upsilon_aV^b - V_a\Upsilon^b+ V^c\Upsilon_c \delta_a^{~b}, \\
\hat\nabla_a\mu_b &=& \nabla_a\mu_b +(w-1)\Upsilon_a\mu_b - \mu_a\Upsilon_b+ \mu_c\Upsilon^c g_{ab},
\end{array}
\end{equation} 
where $\Upsilon_a=\Omega^{-1}\nabla_a\Omega\in \EE_a$ and where the raising and lowering of indices is done with  $g_{ab}$.
As a consequence, for a symmetric rank $\ell$ tensor $\mu_{b_1 \ldots b_\ell}$ of weight $w$, it is 
\[\hat\nabla_a\mu_{b_1 \ldots b_\ell } = \nabla_a\mu_{b_1 \ldots b_\ell} +(w-\ell)\Upsilon_a\mu_{b_1 \ldots b_\ell} -  \Upsilon_{b_1}\mu_{ab_2 \ldots b_\ell }
-\ldots -    \Upsilon_{b_\ell}\mu_{b_1 \ldots b_{\ell-1} a}
+
g_{a(b_1}\rho_{b_2 \dots b_\ell)}
\]
for some rank $\ell-1$ tensor $\rho_{c \ldots d}$, so that
\begin{equation}
\label{trafo-ell}\hat\nabla_{(b_0}\mu_{b_1 \ldots b_\ell )} = \nabla_{(b_0}\mu_{b_1 \ldots b_\ell)} +(w-2\ell)\Upsilon_{(b_0}\mu_{b_1 \ldots b_\ell)} 
+
g_{(b_0b_1}\rho_{b_2 \dots b_\ell)}.
\end{equation}
Hence, for symmetric  tensors of rank $\ell$ and weight $w=2\ell$  the equation that the trace-free part of $ \nabla_{(b_0}\mu_{b_1 \ldots b_\ell)}$ vanishes is conformally invariant.

\subsection{The conformal standard tractor bundle}
Let ${\mathcal J}^1\E[1]$  be the bundle of first order jets of $\E[1]$, which fits canoncically into an exact sequence
\begin{equation}\label{1jet}
\begin{tikzcd}
0\arrow[r] & \EE_{a}[1]  \arrow[r, hook, "\iota_1"]  & {\mathcal J}^1\EE[1] \arrow[r, two heads, "\pi_1"]   &\EE[1]\arrow[r] &0,
\end{tikzcd}
\end{equation}
i.e.~first jets are characterised by the composition series $ {\mathcal J}^1\EE[1]=\EE[1]\ltimes\EE_{a}[1]$. Here we borrow the notation from the semidirect product of groups or semidirect sum of Lie algebras, and by which we mean that $\EE_a[1] $ is a subbundle of $ {\mathcal J}^1\EE[1]$ and $ {\mathcal J}^1\EE[1]/\EE_a[1]$ is isomorphic to $\EE[1]$.
Moreover, let ${\mathcal J}^2\E[1]$  be the bundle of second order jets of $\E[1]$ with the jet exact sequence
\begin{equation}\label{2jet}
\begin{tikzcd}
0\arrow[r] & \EE_{(ab)}[1]  \arrow[r, hook, "\iota_2"]  & {\mathcal J}^2\EE[1] \arrow[r, two heads, "\pi_2"]   &{\mathcal J}^1\EE[1]\arrow[r] &0,
\end{tikzcd}
\end{equation}
so that we have the composition series \[ {\mathcal J}^2\EE[1]= {\mathcal J}^1\EE[1]\ltimes \EE_{(ab)}[1]=\EE[1]\ltimes\EE_{a}[1] \ltimes \EE_{(ab)}[1],\]
noting that $\ltimes$ is associative (but not commutative).
The conformal standard tractor bundle is  defined as the quotient
\[\E_A:={\mathcal J}^2\EE[1]/ \EE_{(ab)_0}[1],\]
where $ \EE_{(ab)_0}$ denotes the trace-free symmetric bilinear forms on $TM$.  Using the conformal metric 
$\mathbf{g}_{ab}$, which is a section of $\EE_{(ab)}[2]$, we have the isomorphism
\[ \E[-1] \ni \rho \longmapsto [\rho\g_{ab}] \in \EE_{(ab)}[1] /  \EE_{(ab)_0}[1].\]
Together with the injection  $ \EE_{(ab)}[1] /  \EE_{(ab)_0}[1]\hookrightarrow \EE_A$, induced from $\iota_2$, this gives  an injection
\[
\begin{tikzcd}
X_A:\E[-1]\arrow[r,hook] &\E_A.\end{tikzcd}\]
Using the canonical isomorphism $\E[w]^*\simeq \E[-w]$,   $X_A$ is a section of $\E[1]\otimes \E_A=:\E_A[1]$. 
In addition, $\pi_2$ descends to $\pi_2: \E_A \to \mathcal J^1\EE[1]$, remains surjective, and it is $\E_A/\E[-1]\simeq \mathcal J^1\EE[1]$, so that we have the exact sequence $\EE[-1] \hookrightarrow \EE_A \twoheadrightarrow \EE[1]$ and hence the composition series
\begin{equation}
\label{filtration}
\EE_A= J^1\EE[1]\ltimes \E[-1]=\EE[1]\ltimes \EE_a[1]\ltimes \E[-1].\end{equation}
From the exact sequence (\ref{2jet}), together with the first jet exact sequence (\ref{1jet}), 
we get
 \[X^A:=\pi_1\circ\pi_2:\E_A\twoheadrightarrow \E[1],\] i.e.~$X^A$ is a section of $ \E_A^*\otimes \E[1]=:\E_A^*[1]$ with 
$X^AX_A=0$.

On $\E_A$ there is a metric $g^{AB}$, a section of $\E_{AB}^*$, that
will provide us with an isomorphisms between $\E_A$ and
$\E_A^*=:\E^A$. It is defined as follows: on $\EE_a[1]\ltimes
\EE[-1]=\pi_2^{-1}(\iota_1(\E_a[1]))\subset \E_A$, the metric
$g^{AB}\in \EE^*_{(AB)}$ is defined by the pull-back $\pi_2^*\g^{ab}$
of ${\g}^{ab}$ by the projection $\pi_2$. Then the image of $X_A$ is
in the kernel of $\pi_2^*\g^{ab}$. Hence we set \[ g^{AB}X_A:=X^B,\]
so that $g^{AB}$ is non-degenerate and conformally invariant, by
construction. In fact, its signature is $(r+1,s+1) $ if the conformal
class is of signature $(r,s)$.  The image of $X_A$ is null for
$g^{AB}$, $g^{AB}X_AX_B=X^AX_A=0$, with $\E[-1]^\perp=\EE_a[1]\ltimes
\EE[-1]
$.

Now we can explain how the Levi-Civita connection $\nabla_a$ of a metric $g_{ab}$ in the conformal class determines a splitting of the conformal tractor bundle. Consider the homomorphism from ${\mathcal J}^2\EE[-1]$ to $\EE[1]\oplus\EE_a[1]\oplus\EE[-1]$ that is defined by assigning to a $2$-jet, represented by a local section $\sigma$, of $\EE[1]$ the triple
\[
\D_A\sigma|_p:=
\left(
\begin{array}{c}
\sigma
\\
\nabla_a\sigma
\\
-\frac{1}{n}(\Delta +\mathsf{J})\sigma
\end{array}\right)_{|_p}\in \EE[1]\oplus\EE_a[1]\oplus\EE[-1]|_p
,\]
where  $\Delta=\nabla^a\nabla_a$ is the  Laplacian 
of $g_{ab}$ and $\J=\ro_a{}^a$ defined in Section~\ref{conventions}.
This homomorphism is clearly surjective and its kernel is $\EE_{(ab)_0}[1]\subset {\mathcal J}^2\EE[1]$, so that it provides an isomorphism to the conformal tractor bundle
\[\EE_A\simeq \EE[1]\oplus\EE_a[1]\oplus\EE[-1].\]
By restricting to the each summand, the inverse of this isomorphism can  be written as $Y_A +Z_A{}^a+X_A$, where $X_A:\E[-1]\to\E_A$ as before and 
\[
Y_A:\EE[1]\to \EE_A,\quad Z_A{}^a:\EE_a[1]\to \EE_A,\]
so that  $Y_A\in \E_A[-1]$ and $Z_A{}^a\in \E_A\otimes \E^a[-1]$ and
\begin{equation}
\label{fundeqs}
X^AY_A=1,
\quad g^{AB} Z_A{}^aZ_B{}^b=\g^{ab},
\quad
X^AZ_A{}^b=Y^AZ_A{}^a=0,\quad Y_AY^A=0.
\end{equation}
With such a splitting we can write tractors in $\E_A$ as column vectors 
\[\E_A=\EE[1] \oplus \EE_a[1]\oplus \EE[-1]\ \ni\ \sigma Y_A +\mu_aZ_A{}^a +\rho X_A=
\begin{pmatrix} \sigma \\ \mu_a\\\rho\end{pmatrix},\]
and the tractor metric as 
\begin{equation}\label{gAB} g_{AB}=2X_{(A}Y_{B)}+ Z_A{}^aZ_B{}^b\,\g_{ab}.\end{equation}
When changing the metric to $\hat{g}_{ab}=\Omega^2g_{ab}$ the splitting changes as follows
\[
\hat X_A=X_A,\quad \hat{Z}_A{}^a={Z}_A{}^a+\Upsilon^aX_A,\quad \hat{Y}_A=Y_A -\Upsilon_aZ_A{}^a- \frac{1}{2}\Upsilon_a\Upsilon^a X_A,\]
or equivalently the components change as follows,  where again $\Upsilon_a=\Omega^{-1}\nabla_a\Omega$,
\[
 \begin{pmatrix} \hat\sigma \\ \hat\mu_a\\\hat\rho\end{pmatrix}=  \begin{pmatrix} \sigma \\ \mu_a+\Upsilon_a\sigma \\\rho-\Upsilon_b\mu^b -\frac{1}{2}\Upsilon_b\Upsilon^b\sigma\end{pmatrix}.\]
Note that although we will usually fix a scale $\sigma$, and hence a
Levi-Civita connection to split the tractor
bundles, as discussed above, we typically will not use $\sigma$ to
trivialise the density bundles, unless otherwise stated.

Note also that the filtration~(\ref{filtration}) of the standard
tractor bundle $\EE_A$ extends to filtration of the tensor products
$\EE_{A_1\ldots A_\ell}=\otimes^\ell \EE_A$ in a natural way.

 \subsection{The normal conformal tractor connection and the Thomas-$\D$-operator}
When fixing a metric $g_{ab}\in \cc$, the conformally invariant normal conformal tractor connection on $\E_A$ is given by
\[\nabla_a \begin{pmatrix} \sigma \\ \mu_b\\\rho\end{pmatrix}= \begin{pmatrix} \nabla_a\sigma-\mu_a  \\ \nabla_a \mu_b  +\rho\bg_{ab}+ \sigma \ro_{ab}\\\nabla_a \rho-\ro_{ab}\mu^b\end{pmatrix}.\]
Here $\nabla_a$ on the right-hand-side of the equation is the
Levi-Civita connection of $g_{ab}$ and $\ro_{ab}$ is the Schouten
tensor.  The normal conformal tractor connection induces a connection
on $\E^A$ and on all tensor products of of $\E_A$ and $\E^A$.  The
metric $g^{AB}$ is parallel with respect to the induced connection on
$\E^{(AB)}$.  The following fundamental identities will be useful,
\begin{equation}\label{nabXYZ}
\begin{array}{rcl}
\nabla_aX_A&=&\bg_{ab} Z_A{}^b,
\\
\nabla_aZ_A{}^b&=&
-\ro_{a}^{~b}X_A-\delta_a^{~b}Y_A,
\\
\nabla_aY_A&=&\ro_{ab}Z_A{}^b.
\end{array}
\end{equation}
Another important conformally invariant operator is the {\em Thomas-$\D$ operator}. If $\mathcal{V}[w]$ is some tensor product of $\E_A$ of weight $w$, then the differential operator
$\tilde{\D}:\Gamma(\mathcal{V}[w])\to \Gamma(\E_A\otimes \mathcal{V}[w-1])$ is defined by 
\[
\tilde\D_AW
= w(n+2w-2)Y_AW 
+ (n+2w-2)Z_{A}{}^a\nabla_aW
-X_A(\Delta +w\mathsf{J})W.
\]
Here $\Delta=g^{ab}\nabla_a\nabla_b$ is the Laplacian with respect to
the conformal tractor connection coupled with the Levi-Civita
connection.  Since we will work with weights $w\not=\frac{2-n}{2}$, we
will use the {\em normalised Thomas-$\D$ operator},
\begin{eqnarray*}
\hD_A W&:=&\tilde \D_A \tfrac{1}{n+2w-2} W
\\
&=& wY_AW 
+Z_{A}{}^a\nabla_aW
-\tfrac{1}{n+2w-2}X_A(\Delta +w\mathsf{J})W
\  = \
\begin{pmatrix}
w W
\\
\nabla_aW
\\
-\frac{1}{n+2w-2}(\Delta +w\mathsf{J})W,
\end{pmatrix}
\end{eqnarray*}
It satisfies
that \begin{equation}
\label{Dfund}
\D_Ag_{BC}=0,
\quad
\D_AX_B=g_{AB},\quad X^A\D_AV=w W,\end{equation}
as well as
\begin{eqnarray*}
\hD_AY_B&=&
-Y_AY_B+ \ro_{ab}Z_A{}^aZ_B{}^b -\tfrac{\nabla^a\ro_{ab}}{n-4} X_A Z_B{}^b+\tfrac{2\J}{n-4} X_{A}Y_{B}+\tfrac{ \ro_{ab}\ro^{ab}}{n-4} X_{A}X_{B}\\
&=&
\begin{pmatrix}
-1 & 0 & 0
\\
0& \ro_{ab} & 0
\\
\tfrac{2\J}{n-4}& \tfrac{\nabla^a\ro_{ab}}{n-4}
&\tfrac{\ro^{ab}\ro_{ab}}{n-4}
\end{pmatrix}
\end{eqnarray*}
and, because of (\ref{nabXYZ}),
\begin{eqnarray*}
\hD_AZ_B{}^b
&=&
-2Y_{(A}Z_{B)}{}^b
-2 \ro^b{}_a X_{(A}Z_{B)}{}^a
+\tfrac{(n-2)\ro^b{}_a +\J\delta^b{}_a}{n-4} X_A Z_B{}^a
+\tfrac{\nabla_a\ro^{ab}}{n-4}X_AX_B\\
&=&
\begin{pmatrix}
0 & -\delta^b{}_a & 0
\\
-\delta^{ab}& 0 & -\ro^{ba}
\\
0
&\tfrac{2 \ro^{b}{}_a+\J \delta^b{}_a}{n-4}&\tfrac{\nabla_a\ro^{ab} }{n-4}
\end{pmatrix}.
\end{eqnarray*}
 Note that 
\[X^A\hD_{(A}Z_{B)}{}^b=-Z_{B}{}^b,\qquad X^A\hD_{(A}Y_{B)}=-Y_B+\tfrac{\J}{n-4} X_B\]
and 
\[X^A\hD_{[A}Z_{B]}{}^b=0,\qquad X^A\hD_{[A}Y_{B]}=-\tfrac{\J}{n-4} X_B.\]
Instead of the Leibniz rule, the operator $\hD$ satisfies
\begin{equation}
\label{leibnizhat}
\hD_A(V\otimes W)=\hD_AV\otimes W+V\otimes \hD_AW- \tfrac{2}{n+2(v+w)-2} X_A\left( \hD_BV \otimes \hD^B W \right),
\end{equation}
and hence 
\begin{equation}
\label{DhatX}
\big[ \hD_A ,X_B\big]
=g_{AB} -\tfrac{2}{n+2w} X_A\hD_B.\end{equation}
Using (\ref{nabXYZ}) one can commute $X_A$ past the $\Delta$ to get the ``shifted'' formula for $\D$,
\[\hD_AW
= \frac{(n+2w)}{n+2w-2}\left( (w-1) Y_AW 
+ Z_{A}{}^a\nabla_aW
- \tfrac{1}{n+2w}(\Delta +(w+1) \mathsf{J})(X_A W)\right).
\]
Using this shifted formula as well as (\ref{nabXYZ}) and (\ref{Dfund}) as straightforward calculations yields that
\begin{equation}
\label{Dcontract}
\D^A(X_A W) =\frac{(n+w)(n+2w+2)}{n+2w} W\quad\text{ and }\quad
\D^A\D_A W=0.\end{equation}
Finally, when the conformal class contains
a flat metric, then $\left[\D_A,\D_B\right]= 0$, see \cite{Gover06,Gover-Peterson-Lap}
for explicit formulae for $\left[\D_A,\D_B\right]$.

\subsection{Scales, scale tractors, and Einstein scales}\label{scalesec}
We call a section $\sigma$ of $\EE[1]$ a {\em scale} if it is
non-vanishing on an open dense set.  We will say that a scale $\sigma$
is a positive/negative scale if $\sigma>0$/$\sigma<0$ everywhere on
$M$. We also use the notation $\E_{\pm}[1]$ for positive and negative
scales. Positive and negative scales define metrics
$g_{ab}=\sigma^{-2}\g_{ab}$ in the conformal class. When we say that a
{\em scale} defines a metric in this way, we mean the metric defined
on the open dense set where the scale does not vanish.

Metrics in the conformal class are in one-to-one correspondence with sections of $\E_+[1]$ via $g_{ab}=\sigma^{-2}\g_{ab}$.
Given a scale $\sigma$, we define corresponding {\em the scale tractor}  as 
$$
I_A:=\hD_A \si \stackrel{g}{=}
\sigma Y_A+\nabla_a\sigma Z_A{}^a -\frac{1}{n}(\Delta \si + \J\si)X_A=
 \left(\begin{array}{c} \si\\ \nabla_a \si \\
-\frac{1}{n}(\Delta \si + \J\si)
\end{array}\right)\in \Gamma(\EE_A).
 $$
If $I_A:=\hD_A \si$ and has no zeros, then $\sigma=X^AI_A$ is a scale.
 We call a scale tractor {\em positive} if $X^AI_A>0$ everywhere.
 Note that, $\iota:=I^AI_A=-\tfrac{2\J}{n}$
 where here $\J=P_{ab}g^{ab}$ is the
 (weight zero) Schouten trace for the metric $g$ determined by
 $\si=X^A I_A$, so up to a negative  constant $I^AI_A$ captures the scalar curvature.
Also observe,
\[\nabla_aI_B= \left(\sigma \ro_{ab}+\nabla_a\nabla_b \sigma-\frac{1}{n} (\Delta+\J)\sigma g_{ab}\right) Z_B{}^b
-\left( \ro_a{}^b\nabla_a\sigma +\frac{1}{n}
\nabla_a\left(\Delta+\J)\sigma\right)\right) X_B,\] from which it
follows that $\sigma^{-2}\g_{ab}$ is Einstein if and only if $I_B$ is
parallel for the tractor connection $\nabla_a$, or equivalently, since
$I_B$ is of weight $0$, if $\D_AI_B=0$.  Any parallel tractor $I_A$ is
$\D_A\sigma$ for $\sigma$ a section $\E[1]$ and so if $I_A\neq 0$ then
$\si$ is a scale. We call such $\si$ an {\em Einstein scale}. See
\cite{gover-nurowski04,Gover10} for further details on the facts
here. Note that here we use Einstein scale here to mean a scale
(corresponding to a parallel tractor) that can have a non-trivial zero
locus. In \cite{Gover10} and many other sources such scales are called
``almost Einstein'' since the metric is not defined at points where
$\si$ vanishes.

\subsection{The projective tractor bundle and Killing tensors}
\label{projsec}
Let $\Lambda^nT^*M$ be the bundle of $n$-forms on an orientable
$n$-dimensional manifold, and for
$w\in\mathbb{R}$, set 
$\E(w)=\left(\Lambda^n\right)^{-\tfrac{w}{n+1}}$. These are densities of {\em
  projective weight $w$}.  Then let $\bar\E_A={\mathcal J}^1\E(1)$ be
the bundle of first jets of $\E(1)$. We call this the {\em projective
  tractor bundle} and we have the first jet exact sequence
\[
\begin{tikzcd}
0\arrow[r] & \EE_{a}(1)  \arrow[r, hook, "\bar Z_A{}^a"]  & \bar\EE_A \arrow[r, two heads, "\bar X^A"]   &\EE(1)\arrow[r] \arrow[l, bend left, "
\bar Y_A"]\arrow[r]  &0,
\end{tikzcd}
\]
where we denote as before $\E_a(w)=\E_a\otimes \E(w)$.
Fixing a connection in the projective class gives a splitting $\bar Y_A$ as in the diagram.
As in the conformal setting, we write projective tractors as columns
\[
 \sigma \bar Y_A +\mu_a\bar Z_A{}^a=
\begin{pmatrix} \sigma \\ \mu_a\end{pmatrix},\]
so that $\bar X^A=(1,0)\in \E^A(1)$. 
 With respect to a splitting that is given by a connection with symmetric Ricci tensor $R_{ab}$, which always exists locally in the projective class, the {\em projective tractor connection} is defined as
\[\bar\nabla_a \begin{pmatrix} \sigma \\ \mu_b\end{pmatrix}= \begin{pmatrix} \nabla_a\sigma-\mu_a  \\ \nabla_a \mu_b  +\tfrac{\sigma}{n-1}R_{ab}
\end{pmatrix}.\]
Following \cite{GoverMacbeth14}, we will now compare the projective
tractor connection with the conformal tractor connection in the case
when there is an Einstein metric $g=\eta^{-2}\bg$, given by the scale $\eta$ and with $\J\not=0$, in the conformal
class.  In this section we will trivialise all density bundles and split the conformal tractor bundle with respect to the
Einstein scale $\eta$, so that the scale tractor is given as
$$
I_A\stackrel{g}{=}
 \left(\begin{array}{c} 1\\ 0 \\
-\frac{\J}{n} 
\end{array}\right).
 $$
We denote by $I_A^\perp$ its orthogonal complement,
\[I_A^\perp=\left\{ 
 \left(\begin{array}{c} \si\\ \mu_a  \\
\frac{\J\si }{n}
\end{array}\right)\mid \si\in \E[1], \mu_a\in \E_a[-1]\right\},
 \]
 where we continue to work in the splitting, 
and we denote by $\pi_A{}^B:\E_A\to I_A^\perp$ the orthogonal projection that maps
\[
 \left(\begin{array}{c} \si\\ \mu_a  \\
\rho
\end{array}\right)
\longmapsto 
 \left(\begin{array}{c} \tfrac{1}{2}\left( \si +\tfrac{n}{\J}\rho\right)  \\ \mu_a  \\
\tfrac{1}{2}\left( \rho +\tfrac{\J}{n}\si\right)
\end{array}\right).
 \]
Equivalently, 
\begin{equation}\label{piAB}
 \pi_A{}^BY_B=\tfrac{1}{2}\left( Y_A+\tfrac{\J}{n}X_A \right) =\tfrac{\J}{n}\pi_A{}^BX_B,\quad  \pi_A{}^BZ_B{}^b=Z_B{}^b.\end{equation}
We consider the projective tractor connection that is defined by the projective class of the Levi-Civita connection of the Einstein metric. It is easily verified  that the bundle isomorphism  between $I_A^\perp$ and $\bar \E_A$,  defined by 
\[
 \left(\begin{array}{c} \si\\ \mu_a  \\
\frac{\J\si }{n}
\end{array}\right)\longmapsto 
 \left(\begin{array}{c} \si\\ \mu_a \end{array}\right),
\]
preserves the connections. This means that the diagram
\[
\begin{array}{rcl}
I_A^\perp\ni 
\left(\begin{array}{c} \si\\ \mu_b  \\
\frac{\J\si }{n}
\end{array}\right)
&\longmapsto &
 \left(\begin{array}{c} \si\\ \mu_b \end{array}\right)\in \bar{\E}_A
 \\[6mm]
 \nabla_a\downarrow\quad&&\quad\downarrow\bar\nabla_a
 \\[3mm]
 \left(\begin{array}{c} \nabla_a\si-\mu_a \\ 
 \nabla_a\mu_b+\tfrac{2\J\sigma }{n}g_{ab}  \\
\tfrac{\J}{n}\left(\nabla_a\sigma -\mu_a\right)
\end{array}\right)
&\longmapsto &
 \left(\begin{array}{c} \nabla_a\si-\mu_a \\ \nabla_a\mu_b +\tfrac{R\sigma}{n(n-1)}g_{ab}\end{array}\right)
\end{array}
\]
commutes, because $2(n-1)\J=R$.
Note that under this isomorphism we have
\begin{equation}
\label{conf-proj}\pi_A{}^BY_A=\tfrac{1}{2}\left( Y_A+\tfrac{\J}{n} X_A\right)\longmapsto\tfrac{1}{2}\bar Y_A,\qquad
\pi_A{}^BX_A=\tfrac{1}{2}\left( X_A+\tfrac{n}{\J} Y_A\right)\longmapsto\tfrac{n}{2\J}\bar Y_A.\end{equation}

\section{Conformal Killing vectors} 
\label{rank1-sec}

\subsection{Prolongation of the conformal Killing vector equation}

Let $(M,\cc)$ be a conformal manifold.  Recall that a vector field $k^a\in \EE^a$
is a {\em conformal Killing vector} if
 \[ g_{c(a}\nabla_{b)_0}k^c =0,\]
 where $\nabla$ is the Levi-Civita connection of $g_{ab}$ from $\cc$. 
 The equation is equivalent to 
 \[  g_{c(a}\nabla_{b)}k^c=\tfrac{1}{n}\nabla_c k^c g_{ab}.\]
  From (\ref{trafo}) it follows that this equation is conformally invariant.  When lowering the index on $k^a$ using the conformal metric $\mathbf{g}_{ab}$, we obtain a one-form of weight $2$, $k_a=\mathbf{g}_{ab}k^b\in \EE_a[2]$ that satisfies the equation 
  \[\nabla_{(a}k_{b)_0}=0,\]
  or equivalently 
  \[ \nabla_{(a}k_{b)}+\rho\bg_{ab}=0,\]
  where $\nabla$ is the Levi-Civita connection of a metric in the conformal class and $\rho\in \EE$. Again, (\ref{trafo-ell}) shows that  this equation is conformally invariant.  Such a $k_a\in \EE_a[2]$ we call a {\em conformal Killing $1$-form}. Taking the trace shows that $\rho $ is determined as the divergence of $k^a$,
\[
    \rho=-\tfrac{1}{n}\nabla_ak^a.
\]
  Hence, a conformal Killing $1$-form is Killing if its divergence vanishes.
  Moreover, we have
  \begin{lemma}\label{DKlemma}
  If $k_a\in \EE_a[2]$ is a conformal Killing $1$-form with 
$\rho=-\frac{1}{n}\nabla_ak^a$, then 
\begin{equation}\label{DK0b}
\Delta k_b=
-\mathsf{J}k_b
-(n-2)k^c\ro_{cb}+(n-2)\nabla_b\rho,
\end{equation}
and
\begin{equation}\label{DK00}
\Delta\rho
= 
-2\J\rho +k^a\nabla_a\mathsf{J}.\end{equation}
  \end{lemma}
  \begin{proof} 
Using that $k_b$ is conformally Killing and commuting derivatives by
introducing curvature,
\begin{eqnarray*}
\Delta k_b&=&-\nabla^a\nabla_bk_a+\tfrac{2}{n} \nabla_b\nabla_ak^a
\\
&=&-R_{ab}{}^{a}{}_ck^c-\nabla_b\nabla_ak^a+\tfrac{2}{n} \nabla_b\nabla_ak^a
\\
&=&
-(n-2)k^c\ro_{cb}-\mathsf{J}k_b-\frac{n-2}{n}\nabla_b\nabla_ak^a.
\end{eqnarray*}
This is equation (\ref{DK0b}). Moreover, differentiating this equation, commuting derivatives and using $\nabla_a\J =\nabla^b\ro_{ab}$ implies that 
\begin{equation}\label{nabdeltak}
\begin{array}{rcl}
\nabla_b\Delta k^b
&=&
-\tfrac{n-2}{n} \Delta \nabla_bk^b 
-(n-2) \ro^{ab}\nabla_ak_b -\J \nabla_bk^b
-(n-1)k^b \nabla_b \J
\\[2mm]
&=&
-\tfrac{n-2}{n} \Delta \nabla_bk^b 
-\tfrac{2(n-1)}{n} \J \nabla_bk^b
-(n-1)k^b \nabla_b \J,
\end{array}\end{equation}
where the second equation uses that $k_b$ is conformally Killing.
Now observe that for any vector $k^b$, not necessarily Killing, we can commute derivatives to get
\[\nabla_b\Delta k^b=\Delta \nabla_bk^b +(n-1)k^b \nabla_b \J+(n-2)\ro^{ab}\nabla_ak_b+\J \nabla_bk^b,\]
and hence, when $k_b$ is conformally Killing, that
\[\nabla_b\Delta k^b=\Delta \nabla_bk^b +(n-1)k^b \nabla_b \J+\tfrac{2(n-1)}{n}\J \nabla_bk^b.\]
This, when substituted into (\ref{nabdeltak}), yields
\[
\Delta \nabla_bk^b=-2\J\nabla_bk^b - nk^b\nabla_b \J,\]
which is 
(\ref{DK00}).
    \end{proof}
  
  We will now give a tractorial characterisation of this equation. In
  the following we raise and lower indices only with the conformal
  metric ${\g}_{ab}$ and with the tractor metric $g_{AB}$.
  There is a conformally invariant operator from conformal Killing
  vectors, or conformal Killing $1$-forms, to sections of the weighted
  tractor bundle $\EE_A[1]$ defined as follows,
\begin{equation}\label{Kdef}
k_a\longmapsto K_A=k_aZ_A{}^a -\frac{1}{n}\nabla_ak^aX_A= \begin{pmatrix} 0\\k_a\\-\frac{1}{n}\nabla_ak^a,\end{pmatrix}\in \EE_A[1],
\end{equation}
where the right hand side uses a metric $g_{ab}$ from $\cc$ to give the splitting, and $\nabla_a$ is the corresponding Levi-Civita connection.
We shall call $K_A$ as in \eqref{Kdef} the {\em half prolongation} of $k_a$.
Next we define 
$$
\K_{AB}=\D_{[A}K_{B]}\in \EE_{[AB]} 
$$
where $\D$ is the normalised Thomas-$\D$-operator, and square
brackets $[ \ldots ]$ denote the skew symmetrisation over enclosed
indices.  For convenience, we refer to $ \K_{AB}$ as the {\em full
  prolongation} of $k_a$ (although the fully prolonged system is an
equation of parallel transport for $ \K_{AB}$).  This is important as
it links to the fully prolonged system, via the result from
\cite{Gover06} that we recover here.
\begin{proposition}\label{KVprop}
Let $k^a$ be a conformal Killing vector and $K_A$ the associated half
prologation in $\EE_A[1]$, as given by \eqref{Kdef}. Then $\D_AK_B$
is skew, i.e.
$$
\D_{A}K_{B}=\K_{AB} .
$$
Conversely, let $K_A\in \Gamma(\EE_A[1])$ such that $K_AX^A=0$  and $\D_{(A}K_{B)}=0$. Then $k_a=Z^{~A}_aK_A$ is a conformal Killing $1$-form.
\end{proposition}
\begin{proof} 
  Assume that $K_A=(0,k_a,\rho)\in \EE_A[1]$. Then
$\D_AK_B$ is given in matrix form as 
\begin{eqnarray*}
\D_AK_B&=&
\begin{pmatrix}
K_B
\\
\nabla_aK_B
\\
-\tfrac{1}{n}(\Delta+\mathsf{J})K_B
\end{pmatrix}
\\
&=&
\begin{pmatrix}
0&k_b&\rho\\
 -k_b& \nabla_ak_b+\rho\bg_{ab} &\nabla_b\rho -k^a\ro_{ba}
 \\
   \tfrac{2}{n} \nabla_ak^a+\rho &\tfrac{1}{n}\left(2k^a\ro_{ba} -\Delta k_b-2\nabla_b\rho-Jk_b\right) & Y^AY^B\D_AK_B 
\end{pmatrix},
\end{eqnarray*}
where $Y^AY^B \D_AK_B$ is computed to be 
\begin{equation}
\label{ydky}
Y^AY^B\D_AK_B=\tfrac{1}{n} \left( 2\nabla_a k_b\ro^{ab}+k^b\nabla_b\J-\Delta\rho\right).\end{equation}
For the forward implication, assume that $k_a$ is a conformal Killing $1$-form and that $\rho=-\frac{1}{n}\nabla_ak^a$. For $\D_AK_B$ to be skew it remains to show that
\[ 0=  \Delta k_b+\tfrac{(n-2)}{n}\nabla_b\nabla_ak^a+(n-2) k^a\ro_{ba}+Jk_b,\]
and that 
\[
0=Y^AY^B\D_AK_B.
\]
The first equation is
equation (\ref{DK0b}) in Lemma \ref{DKlemma}. When $k_b$ is conformally Killing, (\ref{ydky}) gives
\[
Y^A Y^B \D_AK_B=\tfrac{1}{n}\left( k^b\nabla_b\J-\Delta\rho-2\J  \rho\right),\]  
which vanishes because of 
(\ref{DK00}).

For the converse direction note that if $\D_{(A}K_{B)}=0$, then the matrix form shows that 
$\nabla_{(a}k_{b)}+\rho\bg_{ab}=0$ and hence $k_b$ is a conformal Killing $1$-form.
\end{proof}

With $\D_BK_C$ skew, in the conformally flat case the commutator $[\D_A,\D_B]$ vanishes on weighted tractors, see \cite{Gover06,Gover-Peterson-Lap}, and so we have that  
\[\D_{A}\D_{B}K_C\in ( \E_A\otimes \E_{[BC]})\cap ( \E_{(AB)}\otimes \E_C) =\{0\} \]This implies
\[
0=
\D_{A}\D_{B}K_C=
Z_{A}{}^a\nabla_a\D_BK_C
-\frac{1}{n-2}X_A\Delta \D_BK_C,\]
so that 
\[
0=Z^{Ab}\D_{A}\D_{B}K_C=\nabla_b\D_BK_C,
\]
i.e.~$\D_BK_Z$ is a section of $\E_{[BC]}$ that is parallel for the normal conformal tractor connection.

\subsection{Killing scales for conformal Killing vectors}
\label{vfKscale}
Let $k^a$ be a conformal
Killing vector. We say that a scale $\sigma\in \EE[1]$ (as defined in
Section~\ref{scalesec}) is a {\em Killing scale for $k^a$} if $k^a$ is
a Killing vector for the metric $g_{ab}=\sigma^{-2}\bg_{ab}$. This is
equivalent to the existence of a metric $g_{ab}$ in $\cc$ for which
$k^a$ is divergence free, $\nabla_ak^a=0$, for $\nabla$ the Levi-Civita
connection of $g_{ab}$.  In regards to Killing vectors we immediately
obtain the following. 
\begin{proposition}\label{Kscalevec}
  Let $k^a$ be a conformal Killing vector, with associated  half prolongation
  $K_A$. A scale $\sigma$  is a Killing scale for $k^a$ if and only if 
$$
I^AK_A=0,
$$
 where $I^A$ is the scale tractor  of $\sigma$. 
\end{proposition}
\begin{proof}
If $\sigma\in \EE[1]$ with corresponding scale tractor $I_A$ we have
\begin{equation}\label{ik} I^AK_A\stackrel{g}{=}k^a\nabla_a\sigma-\tfrac{\sigma}{n}\nabla_ak^a,
\end{equation}
where we work in a metric $g\in\cc$.

In particular if we take $g$ to the metric
$g_{ab}=\sigma^{-2}\bg_{ab}$, then $\nabla^g\si=0 $ and (also
trivialising density bundles using $\si$) we have
$I^AK_A=-\frac{1}{n}\nabla_ak^a$, which shows the equivalence.
\end{proof}
We recover a well-known fact about conformal Killing vectors.
\begin{corollary}\label{aKscale}
If $k^a$ is a conformal Killing vector, then there is a Killing scale
on the complement of the zero locus of $k_ak^a$. In particular, in
Riemannian signature, there is a dense 
subset where $k^a$ is a
Killing vector for a metric from the conformal class.
\end{corollary}
\begin{proof}
Set $\sigma= \sqrt{|k_ak^a|}= \sqrt{\epsilon k_ak^a}\ge 0$ where $\epsilon$ is the sign of $k_ak^a$. On $\{\sigma\not=0\}$ we have
\[
\nabla_b\sigma=\epsilon \sigma^{-1} k^a\nabla_b k_a,
\]
where we work in the scale of some metric $g\in\cc$.  
With $k^a$ conformal
Killing, this implies
\[k^b\nabla_b\sigma = \epsilon \sigma^{-1}k^bk^a\nabla_bk_a
=- \epsilon\sigma^{-1} k^bk^a\rho g_{ab}=  \tfrac{\sigma}{n}\nabla_ak^a.\]
Hence, for the scale tractor $I_A$ for the scale $\sigma$, on $\{\sigma\not=0\}$ we get from \eqref{ik} that $I^AK_A=0$.
 Then the result follows from
Proposition \ref{Kscalevec}. 

In the Riemannian setting $k_ak^a$ is nowhere zero on an open dense
set, as otherwise $k_a$ is zero on an open set which implies that it
vanishes everywhere on (connected) $M$, as the Killing equation is a
finite type overdetermined linear PDE, see e.g. \cite{BransonCapEastwoodGover06}. 
\end{proof}

Although not central to our later developments, 
the following observation gives some insight into the nature of the
zero set of a conformal Killing vector. 
\begin{corollary}
Let $k^a$ be a conformal Killing vector and $\sigma\in \EE[1]$ a
Killing scale for $k^a$. Then
\[\mathrm{zero}(k_a)
\subset \mathrm{zero}(\nabla_ak^a)\cup \mathrm{zero}(\sigma).\]
\end{corollary}
\begin{proof}
This follows from
\[0=I^AK_A=-\frac{\si}{n} \nabla_ak^a+ k^a\nabla_a\si.\]
At a zero of $k^a$ we must have that also $\sigma$ or $\nabla_ak^a$ vanishes.
\end{proof}

The next result provides an algebraic condition for a class of Killing scales, namely those that correspond to (almost) Einstein metrics.
\begin{corollary}\label{Ein-cK}
  Let $k^a$ be a conformal Killing vector with full prolongation $\K_{AB}$.
  If  $\sigma\in \EE[1]$ is a Killing scale for $k^a$ that is also an Einstein scale, then  $$
I^A\K_{AB} =0 .
$$
\end{corollary}
\begin{proof}Since $\si$ is an Einstein scale it satisfies $\nabla_aI_B=0$, this gives $\D_AI_B=0$ and by  the modified Leibniz rule (\ref{leibnizhat}) we get $\D_A(I^BK_B)=I^B\D_AK_B$. Therefore  with $I^B$ also being a Killing scale, i.e.~$I^BK_B=0$ we conclude $I^B\D_AK_B=0$. Since $\K_{AB}=\D_A K_B$ is skew the result follows. 

We can also see this directly:
if $k^a$ is a Killing vector we have from equation (\ref{DK0b}) that
\[\Delta k_b +\mathsf{J}k_b=-(n-2)k^a\ro_{ab},\] 
so that in a Killing scale we have 
\[\D_AK_B
=
\begin{pmatrix}
0 & k^a\ro_{ab} & 0
\\
-k^a\ro_{ab} & \nabla_ak_b& -k_b
\\
0 &k_b&0
\end{pmatrix},\]
away from any scale zeros.
Now if $I^A= \left(\begin{array}{c} 1\\ 0\\
-\frac{\sJ}{n}
\end{array}\right)$ is the scale tractor, in the $\si$ scale (and using $\si$ to trivialise density bundles off the zero locus),   we get
\[I^B\D_AK_B=\left(0,  \tfrac{J}n k_b-\ro_{ab}k^b,0\right).\]
If $\si$ is also an Einstein scale   we have $\ro_{ab}=\frac{\mathsf{J}}{n}g_{ab}$, and hence $I^A\D_AK_B=0$.
\end{proof}

\subsection{Conformal versus Killing vectors via projective tractors}
In this section, we restrict to the conformally flat setting, so that
the normal conformal tractor connection is a flat connection and
$\left[\D_A,\D_B\right]=0$. If $k_a$ is a conformal Killing vector, in
the previous section we have seen that $\D_AK_B$ is parallel for the
normal conformal tractor connection. Conversely the top slot of a skew
parallel 2-tractor is a conformal Killing vector field. In
\cite[Section~4.1]{GoverLeistner19} it is shown that similarly on projectively flat
manifolds skew parallel {\em projective} 2-tractors correspond to
Killing vector fields (in fact the prolongation connection in this case is due to Kostant \cite{Kostant55}, but not using tractors). We will use that here.

Assume now that we work in a non-flat Einstein scale $\si$ (and we use $\si$ to trivialise density bundles) with
$\ro_{ab}=\frac{\J}{n} g_{ab}$ and corresponding parallel scale
tractor
$$
I_A=Y_A-\tfrac{\J}{n}X_A.
$$ 
Note that in such a scale the metric $g$ has
constant sectional curvature.  We have, from Proposition \ref{KVprop}
and $\ro_{ab}=\frac{\J}{n}g_{ab}$, that
\[
\D_AK_B= 
2k_b Y_{[A}Z_{B]}{}^b 
+2 \rho Y_{[A}X_{B]}
- 2 (\nabla_b\rho -\tfrac{\J}{n} k_{b}) X_{[A}Z_{B]}{}^b
+ ( \nabla_ak_b+\rho\bg_{ab} )Z_A{}^aZ_B{}^b
.
\]
 Since $I_A$ and $\D_AK_B$ are parallel for the conformal tractor connection, also $I^A\D_AK_B$ is parallel for the conformal tractor connection. We have
\[
I^A\D_AK_B
=
 - \left( \nabla_b\rho \,Z_B{}^b +\rho\left( Y_B+\tfrac{\J}{n} X_B\right)\right),\]
 and hence
 \begin{equation}\label{hessrho}
   0=Z^B{}_b\nabla_a (I^A\D_AK_B)=-\left(\nabla_a\nabla_b\rho +\tfrac{2\J}{n}\rho g_{ab}\right).\end{equation}
This shows that in an Einstein scale, the  Hessian of the divergence of a conformal Killing vector field on a conformally flat manifold is a multiple of the metric.

Now 
we use the relation to the  projective setting in Section \ref{projsec}  to obtain a Killing vector. Let $\pi_A{}^B:\E_B\to I_A^\perp$ denote the projection to the orthogonal complement of $I_A$ and define 
\[
L_{AB}:=\pi_A{}^C\pi_B{}^D\D_CK_D.
\]

By equation (\ref{piAB}) for the projection $\pi_A{}^B$, we get that $\pi_{A}{}^C\pi_{B}{}^DY_{[C}X_{D]}=0$ and therefore
\[
L_{AB}
= 
2\left( 2k_b-\tfrac{n}{\J}\nabla_b\rho\right)\bar Y_{[A}Z_{B]}{}^b
+ ( \nabla_ak_b+\rho\bg_{ab} )Z_A{}^aZ_B{}^b \in I^\perp_{[AB]}
,
\]
where, according to~(\ref{conf-proj}), $\bar Y_A=\pi_A{}^BY_A= \frac{1}{2}\left( Y_A+\frac{\J}{n}
X_A\right)$.  Since $I_A$ is parallel for the conformal tractor
connection, so is $\pi_A{}^B$. Therefore, $L_{AB}$ is parallel for the
conformal tractor connection, and hence, being a section of
$I^\perp_{[AB]}$, is also parallel for the projective tractor
connection $\bar\nabla$.  If we set $\bar X_A=\pi_A{}^BX_A=
\frac{1}{2}\left( X_A+\frac{n}{\J} Y_A\right)$ we have $\bar X^A\bar
Y_A= \tfrac{1}{2}$, and hence the vector field,
\[\bar k_b :=2\bar X^A Z^{B}{}_b L_{AB} = 2 X^A Z^{B}{}_b L_{AB} = 
 2k_b-\tfrac{n}{\J}\nabla_b\rho
\]
is a Killing vector field in the chosen Einstein scale
\cite{GoverLeistner19}. Indeed, from equation (\ref{hessrho}) we can
see how this is working directly:
\[
\nabla_{(a} \bar k_{b)}=2 \nabla_{(a} k_{b)} -\tfrac{n}{\J}\nabla_{(a} \nabla_{b)}\rho
=
-2 \rho g_{ab}
-\tfrac{n}{\J}\nabla_{(a} \nabla_{b)}\rho=0.\]
We arrive at:
\begin{proposition}\label{newkillingvec}
On a conformally flat manifold, 
let $k_a\in \EE_a[2]$ be a conformal Killing $1$-form. Then for any Einstein scale, with Levi-Civita connection $\nabla$ and trace of Schouten $\J$, the vector field
$$
\textstyle k_b +\frac{1}{2J}\nabla_b\nabla_a k^a
$$
is Killing. 
\end{proposition}
We suspect that this result was known classically, but we do not know
a source. It also seems likely that a version of this will apply in
the curved setting, but we have not pursued that here as, althought it
would be interesting, it is not within our main direction.

\section{Conformal Killing tensors of rank two}

\label{rank2-sec}

\subsection{Conformal Killing tensors and the Bertrand--Darboux equation}

Let $(M,\cc)$ be a conformal manifold. Unless stated otherwise, in the following we raise and lower indices with the conformal metric $\mathbf{g}_{ab}\in \Gamma(\ce_{(ab)}[2])$. If $k_{ab}\in \Gamma(\ce_{(ab)}[w])$ is a bilinear form of conformal weight $w$, then for two metrics
$g$ and $\hat{g}=\Omega^2g$  in $\cc$, with a non vanishing function $\Omega$, from~(\ref{trafo-ell}) we get
\[
\hat{\nabla}_{a}k_{bc} =\nabla_{a}k_{bc}+{(w-2)}\Upsilon_a k_{bc}-\Upsilon_b k_{ac}-\Upsilon_ck_{ab}+\Upsilon^dk_{dc}g_{ab}+ \Upsilon^dk_{db}g_{ac},\]
where $\nabla$ is the Levi-Civita connection of a $g\in c$ and \[\Upsilon_a=\Omega^{-1}\nabla_a\Omega\in \ce_a
\]
is of weight $0$.  This shows that if $w=4$ then the equation that
$\nabla_{(a}k_{bc)}$ is pure trace is conformally invariant.  Hence,
we call a symmetric trace-free tensor $k_{ab}\in \Gamma(\E_{(ab)_0}[4])$ a {\em
  conformal Killing tensor (of rank $2$ and weight~$4$)}, if
$\nabla_{(a}k_{bc)}$ is pure trace, i.e.~if
%\begin{equation}
%\label{ckeq}
\[\nabla_{(a}k_{bc)} =\rho_{(a}\bg_{bc)}, \quad \text{ with }\quad \rho_a=\frac{2}{n+2} \nabla_bk^b_{~a}
\in \E_a[2].
\]
If we change to the metric $\hat{g}=\Omega^2g$, in $\cc$, we have
\begin{equation}\label{ro}
  \hat{\rho}_a=\rho_a +2 k_{ab}\Upsilon^b.
\end{equation}
By raising both indices of a conformal Killing tensor $k_{bc}$ (using
$\bg^{ab}$) we obtain a  tensor field $k^{bc} $ of weight 0.
A tensor $\bar k_{ab}\in \Gamma(\ce_{(ab)}[4])$ is called a {\em Killing
  tensor for $g_{ab}$} if $\nabla_{(a}\bar k_{bc)}=0$.
(The weighting of Killing tensors is unusual but is useful here to relate to the conformal objects that arise.)

  Clearly, the trace-free part of a Killing tensor is a conformal Killing tensor. Conversely,  a conformal Killing tensor
 $k_{ab}\in \ce_{(ab)_0}[4]$ is a Killing tensor for
  $g_{ab}\in\cc$ if and only if
   the Levi-Civita connection $\nabla$ of $g_{ab}$ satisfies
  $\nabla_bk^b_{~a}=0$.

For the remainder of this subsection we will consider the following questions. Given $(M,\cc)$ and a conformal Killing tensor $k_{ab}$:
\begin{enumerate}
\item When are there two metrics in the conformal class for which  $k_{ab}$ is Killing? 
 \item When are there one or more metrics  in the conformal class for which $k_{ab}$ is the trace-free part of a Killing tensor $\bar{k}_{ab}$?
\end{enumerate}
 In regards to the first question, directly from (\ref{ro}) we obtain the following answer.
 \begin{lemma}\label{kdvlemma}
 Let $k_{ab}\in \ce_{(ab)_0}[4]$ be a conformal Killing tensor of
 $(M,\cc)$ and  assume that there is a metric $g_{ab}\in \cc$ such that $k_{ab}$
 is a Killing tensor for $g_{ab}$. Then there is another metric $\hat{g}_{ab}$ in $\cc$
 for which $k_{ab}$ is Killing if and only if there is exact one-form
 $\Upsilon_a$ (of weight $0$) such that
 \begin{equation}\label{kdv}
 k_{ab}\Upsilon^b=0.\end{equation}
 The metric $\hat g$ is obtained from $g$ via $\hat g=e^{2V}g$, where $V$ is a function with $\Upsilon_a=(\d V)_a$. 
 \end{lemma}

Now, in regards to the second question, we allow for a trace modifications $
\bar{k}_{ab}=k_{ab}+\lambda \bg_{ab} $, with $ \lambda\in
\mathcal{E}[2]$, of $k_{ab}$ and assume that $\bar{k}_{ab}$ is a
Killing tensor for a specific metric $g\in \cc$, i.e.~that
 \begin{equation} \label{killing1}
 0=
 \nabla_{(a}\bar{k}_{bc)}
 =
 \nabla_{(a}k_{bc)} +\nabla_{(a}\lambda \bg_{bc)}
 =
 \rho_{(a}\bg_{bc)}
 +\nabla_{(a}\lambda \bg_{bc)}.
 \end{equation}
 This implies that $ \rho_{a}=-
 \nabla_{a}\lambda$. Hence, we get the following.
 \begin{lemma}\label{divkclosed}
 Let 
$k_{ab}\in \ce_{(ab)_0}[4]$ be a conformal Killing tensor of $(M,\cc)$. Then there is 
a metric $g\in \cc$ and a Killing tensor $\bar{k}_{ab}$ for $g$ with trace-free part  $k_{ab}$ 
if and only if there is a $\lambda\in \mathcal{E}[2]$ such that
 \begin{equation}
 \label{divgrad}
\nabla_{a}\lambda=- \frac{2}{n+2} \nabla_bk^b_{~a}, \end{equation}
where $\nabla$ is the Levi-Civita connection of $g$. (In this case $\bar{k}_{ab}-k_{ab}=\lambda \mathbf{g}_{ab}$.) 
 \end{lemma}
 Since the density bundles are flat for all Levi-Civita connections from the conformal class, it implies that the weighted  one-form
 $\rho_a=\frac{2}{n+2} \nabla_bk^b_{~a}
\in \ce_a[2]
$
  is ``closed'', in that
 \[ \nabla_{[a}\rho_{b]}=0.\] 
 In particular, when  trivialising the density bundles with respect to the metric $g_{ab}=\sigma^{-2}\bg_{ab}$, then $f=\sigma^{-2}\lambda$ is  a function and $r_a=\sigma^{-2}\rho_a$ an unweighted $1$-form and we get that  equation (\ref{divgrad}) is equivalent to \[(df)_a=-r_a=-\frac{2\sigma^{-2}}{n+2} \nabla_bk^b_{~a}.\] 
This is equivalent to $\sigma^{-2}\nabla_bk^b_{~a}$ being exact as an unweighted one-form.

  If we further assume that there is a different metric $\hat{g}=\Omega^2g$ in $\cc$ for which another trace adjustment
 $\hat{k}_{ab}=
 k_{ab}+\hat{\lambda} \bg_{ab}$, with $ \hat{\lambda}\in \mathcal{E}[2]$,
 is Killing, we get from (\ref{divgrad}) that 
 \[
 \nabla_{a}\lambda=- \rho_a,\quad  \hat{\nabla}_{a}\hat{\lambda}=- \hat{\rho}_a= - \rho_a -2\ k_{ab}\Upsilon^b. \]
 Since $\hat{\nabla}_{a}\hat{\lambda}={\nabla}_{a}\hat{\lambda}+2\Upsilon_a\hat{\lambda}$, equation
 (\ref{ro}) implies that
 \[\nabla_a(\lambda-\hat{\lambda})=\hat{\rho}_a-\rho_a+2\, \Upsilon_a\hat\lambda= 2\ k_{ab}\Upsilon^b +2\, \Upsilon_a\hat\lambda = 2\ \hat{k}_{ab}\Upsilon^b.\]
Hence we obtain: 
\begin{proposition}
Let $k_{ab}\in \ce_{(ab)_0}[4]$ be a conformal Killing tensor for $(M,\cc)$. Assume that there are two metrics $ g$ and $\hat{g}$ and two Killing tensors $\bar k_{ab}$ and $\hat{k}_{ab}$ with respect to $ g$ and $\hat{g}$, respectively, both with trace-free part $k_{ab}$. Then there is a $\Lambda\in \mathcal{E} [2]$ and a exact one-form $\Upsilon_a$  such that
\begin{equation}
\label{dkdv1}
\nabla_a\Lambda = -\hat{k}^b_{~a}\Upsilon_b\quad \text{ and }\quad
\hat{\nabla}_a\Lambda =\bar{k}^b_{~a}\Upsilon_b,\end{equation}
where $\nabla$ is the Levi-Civita connection of $g$ and $\hat{\nabla}$ that of $\hat{g}$.
\end{proposition}

We are now going to derive the so-called $dkdV$-equation \cite{Smirnov06} by choosing a trivialisation of the density bundles and evaluating the equation
%\begin{equation}\label{dkdv2}
\[
\nabla_a\Lambda = -\hat{k}^b_{~a}\Upsilon_b .
\]
We denote the trivialisations by $\sigma$ and $\hat\sigma$ such that 
\[\bg_{ab}=\sigma^2g_{ab}=\hat\sigma^2\hat g_{ab},\] so that 
\[\hat g_{ab}=\Omega^2 g_{ab}\quad\text{ with }\quad\Omega=\sigma\hat\sigma^{-1}.\]
Multiplying the first equation in  (\ref{dkdv1}) through with
$\sigma^{-2}$, since $\sigma$ commutes with
$\nabla=\nabla^g$, we get to
\[
(df)_a = 
-\sigma^{-2}
\hat{k}^b_{~a}\Upsilon_b
=
-\Omega^{-2}
\hat\sigma^{-2}\hat{k}^b_{~a}\Upsilon_b,
\]
where  $f:=\sigma^{-2}\Lambda$. 
Now we define $\hat K^b_{~a}= \hat\sigma^{-2}\hat{k}^b_{~a}$ the trivialisation of $\hat{k}^b_{~a}$ with respect to $\hat g_{ab}$ of the weight $2$ tensor $\hat{k}^b_{~a}$ and recall that 
$\Upsilon_a= (\d\log \Omega)_a$ to get
%\begin{equation}\label{step2}
\[(\d f)_a = 
-\Omega^{-3}
\hat K^b_{~a}(\d\Omega)_b
=
\frac{1}{2}
\hat K^b_{~a}(\d\Omega^{-2})_b
=
\frac{1}{2}
\hat K^b_{~a}(\d V)_b,
\]
where we set $V=\Omega^{-2}\in C^\infty(M)$.
This is an equation for unweighted quantities, so  we  arrive at the following conclusion.
\begin{corollary}
Let $(M,\cc)$ be a conformal manifold with a conformal Killing tensor $k_{ab}\in \ce_{(ab)_0}[4]$. Assume that there is a metric $g_{ab}\in c$ and a Killing tensor $K_{ab}$ for $g_{ab}$ with trace-free part $k_{ab} $. Then there is another metric $\hat g_{ab}\in c$  and a Killing tensor $\hat K_{ab}$ for $\hat g_{ab}\in c$ with trace-free part $k_{ab} $
if and only if there is a function $V$ such that
%\begin{equation}\label{dkdv}
\[
\d (\hat K^b_{~a}(\d V)_b)=0,
\]
and the metrics are related by $\hat g=V^{-1}g$.

\end{corollary}
In the following sections we will derive tractorial characterisations
of the results above.

\subsection{The half prolongation of the conformal Killing equation}

We would like results for conformal Killing tensors of rank $2$ that
are analogous to the results for conformal Killing vectors.  For this
we will need the half and full prolongations of a conformal
Killing tensor of rank $2$.  As a preparation, we note that a tractor
$K_{AB}\in \E_{(AB)_0}[2]$ satisfies $X^AK_{AB}=0$ if and only if
\begin{equation}\label{K2Xperp}
\begin{array}{rcl}
K_{AB}&=&
k_{ab}{Z}_A{}^a{Z}_B{}^b 
- \rho_a {Z}_{(A}{}^a X_{B)}
+\rho X_AX_B,
\end{array}
\end{equation}
with  $k_{ab}$, $\rho_b$ and $\rho$ sections of $\E_{(ab)_0}[4]$, $\E_a[2]$ and $\E$, respectively.
In addition, we have the following.

\begin{lemma}\label{dklemma1}
If $K_{AB}$ is a section of $\E_{(AB)_0}[2]$ with $X^AK_{AB}=0$, then 
\[X^A \D_{(A}K_{BC)}
=
\tfrac{4}{3(n+4)} X_{(B}\D^AK_{C)A}.\]
In particular,
if 
$\D^AK_{AB}=0$, then $X^A\D_{(A}K_{BC)}=0$.
\end{lemma}
\begin{proof}
  Since $X^AK_{AB}=0$, we have, by (\ref{DhatX}),
\[0
=
\D_B(X^AK_{AC})
=
K_{BC}+X^A\D_{B}K_{AC}-\tfrac{2}{n+4} X_B\D^AK_{AC}.
\]
Hence,
\[
3X^A \D_{(A}K_{BC)}
=
2K_{BC}
+2X^A\D_{(B}K_{C)A}
=
\tfrac{4}{n+4} X_{(B}\D^AK_{C)A},
\]
proving the statement.\end{proof}

\begin{lemma}\label{dklemmaV1}
Let $K_{AB}$ a section of $\E_{(AB)_0}[2]$ with $X^AK_{AB}=0$.  
Then
$\D^AK_{AB}=0$ 
if and only if  $\rho_a$ and $\rho$ in (\ref{K2Xperp}) satsify
\begin{equation}
\label{rhoa-eqn}
 \rho_a= \tfrac{2}{n+2} \nabla^bk_{ab}\end{equation}
and 
\begin{equation}
\label{rho-eqn}\rho= \tfrac{1}{n+1} 
\left( \tfrac{1}{2}
 \nabla^a\rho_a +P^{ab}k_{ab}\right) 
=
 \tfrac{1}{n+1}\left( \tfrac{1}{n+2}
 \nabla^c\nabla^dk_{cd} +P^{cd}k_{cd}\right).
 \end{equation}

\end{lemma}
\begin{proof}
With $K_{AB}$ as in (\ref{K2Xperp}) of weight $2$,
 by definition 
\begin{equation}\label{DK}
\D_AK_{BC}=2Y_AK_{BC}+ Z_A^{~a}\nabla_aK_{BC}-\tfrac{1}{n+2}\left( X_{A}\Delta K_{BC}+2\J X_AK_{BC}\right)
\end{equation}
Since $X^AK_{AB}=0$, we have 
\[\D^AK_{AB}=2Y^AK_{AB}+Z^{Aa}\nabla_aK_{AB}-\tfrac{1}{n+2}X^A\Delta K_{AB}.\]
For the last term we use (\ref{nabXYZ}) to commute $X^A$ past the $\nabla_a$'s to get
\[
0=\Delta (X^AK_{AB}) = -n Y^AK_{AB} +2Z^{Aa}\nabla_aK_{AB}+X^A\Delta K_{AB},\]
so that
\[
\D^AK_{AB}=\tfrac{n+4}{n+2}\left( Y^AK_{AB}+ Z^{Aa}\nabla_aK_{AB}\right)
=\tfrac{n+4}{n+2}\left( Z^{Aa}\nabla_aK_{AB} -\tfrac{1}{2}\rho_{b}Z_B{}^b 
+\rho X_B   \right).
\]
Since 
\[0=\nabla_a g_{bc} =2Z^A{}_b\nabla_a Z_{Ac},\]
we  compute $Z^{Aa}\nabla_aK_{AB}$  as
\begin{eqnarray*}
Z^{Aa}\nabla_aK_{AB}&=&
 \left( \nabla^ak_{ab} -\tfrac{n+1}{2}\rho_b\right)Z_B{}^b
 +\left(n\rho -k_{ab}P^{ab}-\tfrac{1}{2}\nabla^a\rho_a\right) X_B,\end{eqnarray*}
where we also use (\ref{nabXYZ})  again and  that  $k_{ab}$ is trace-free. 
Hence, 
\[
\D^AK_{AB}
=\tfrac{n+4}{n+2}\left( 
 \left( \nabla^ak_{ab} -\tfrac{n+2}{2}\rho_b\right)Z_B{}^b
+\left((n+1) \rho -k_{ab}P^{ab}-\tfrac{1}{2}\nabla^a\rho_a\right) X_B\right).
\]
This shows that 
$\D^AK_{AB}=0$ if and only if equations (\ref{rhoa-eqn}) and (\ref{rho-eqn}) are satisfied. 
\end{proof}

\begin{proposition}
The map
\begin{equation}\label{rk2half}
\begin{array}{rcrcl}
k_{ab}&\longmapsto &
K_{AB}&=&
k_{ab}{Z}_A{}^a{Z}_B{}^b 
- \rho_a {Z}_{(A}{}^a X_{B)}
+\rho X_AX_B, \quad \text{where}
\\
&& \rho_a&:=& \tfrac{2}{n+2} \nabla^bk_{ab},
\\
&&
\rho&:=& \tfrac{1}{n+1} 
\left( \tfrac{1}{2}
 \nabla^a\rho_a +P^{ab}k_{ab}\right) 
=
 \tfrac{1}{n+1}\left( \tfrac{1}{n+2}
 \nabla^c\nabla^dk_{cd} +P^{cd}k_{cd}\right),
\end{array}
\end{equation}
is an isomorphism between sections of $\E_{(ab)_0}[4]$ and sections of $\E_{(AB)_0}[2]$ that satisfy the equations
\begin{equation}\label{xdk2eqs}
X^AK_{AB}=0  \quad\text{ and }\quad
\D^AK_{AB}=0.\end{equation}
If $k_{ab} $ is a conformal Killing tensor, we shall call $K_{AB}$ as in \eqref{xdk2eqs} the {\em half prolongation} of $k_{ab}$.
\end{proposition}
\begin{proof}
Lemma \ref{dklemmaV1} shows that the map is surjective on sections satisfying  equations~(\ref{xdk2eqs}). To show that the map is injective, we note that its inverse is 
\[K_{AB}\longmapsto k_{ab}:=Z^A{}_aZ^B{}_bK_{AB}.\]
This completes the proof.
\end{proof}

From this we obtain to the main result of this section, which yields
the first part of Theorem~\ref{theo0} in the introduction.

\begin{theorem}\label{theo0a}
On a conformal manifold, the isomorphism (\ref{rk2half}) restricts to an isomorphism between conformal Killing tensors $k_{ab}$ and 
sections  $K_{AB}$ of $\E_{(AB)_0}[2]$ with
$X^AK_{AB}=0$ and 
\begin{equation}\label{halfprolong}
\D_{(A}K_{BC)}=X_{(A}F_{BC)},\quad\text{ for some section $F_{AB}$ of $\E_{(AB)_0}[2]$ with $X^AF_{AB}=0$.}\end{equation}
\end{theorem}

\begin{proof} 
First note that from (\ref{DK}) and (\ref{fundeqs}),  for $K_{AB}$ as in (\ref{K2Xperp}), we get
\[
Z^A{}_aZ^B{}_bZ^C{}_c\D_{(A}K_{BC)}
=
Z^A{}_{(a}Z^B{}_bZ^C{}_{c)}\D_{A}K_{BC}
=
 Z^B{}_{(b}Z^C{}_{c}\nabla_{a)}K_{BC}
\]
and
\[
Z^B{}_{b}Z^C{}_{c}\nabla_{a}K_{BC}
=
\nabla_{a}k_{bc} 
-\rho_b 
Z^C{}_{c}\nabla_{a}X_C
=
\nabla_{a}k_{bc} 
-\rho_b 
g_{ac}
\]
so 
\[
Z^A{}_aZ^B{}_bZ^C{}_c\D_{(A}K_{BC)}
=\nabla_{(a}k_{bc)} 
-\rho_{(b }
g_{ac)}.\]
If property (\ref{halfprolong}) holds, this has to vanish and hence $k_{ab}$ is conformally Killing.
Conversely,
if   $k_{ab}$ is conformally Killing this vanishes. Moroever, 
by Lemmas \ref{dklemma1} and~\ref{dklemmaV1}, $\D_{(A}K_{BC)}$ does not contain any $Y_A$ terms, so 
$\D_{(A}K_{BC)}=X_{(A}F_{BC)}$  with $X^AF_{AB}=0$. Taking the trace of this equation, again by Lemma \ref{dklemmaV1}, we get that $F_{AB}$ is trace-free. 
\end{proof}

\subsection{The full prolongation in the conformally flat setting}
\label{theo1-sec}
In this section we  assume that the manifold is conformally flat and we will establish the correspondences in Theorem \ref{theo1} from the introduction. 
The correspondence between (\ref{theo1-1}) and  (\ref{theo1-2})  is obtained as an extension of the results of the previous section, as follows.
\begin{corollary}
On a conformally flat manifold, there is a one-to-one correspondence between conformal Killing tensors and trace-free symmetric tractors $K_{BC}\in \E_{(BC)_0}[2]$  with
$X^AK_{AB}=0$ and 
 $\D_{(A}K_{BC)}=0$.
\end{corollary}

\begin{proof} 
With Theorem \ref{theo0a} it remains to show that for conformally flat manifolds the tractor defined by a conformal Killing tensor satisfies $\D_{(A}K_{BC)}=0$. We will use that different $\D_A$'s commute in the flat case. Since $\D^A\D_A=0$, this implies that  \[3\D^A\D_{(A}K_{BC)}
=2\D_{(B}\D^AK_{C)A}
=0,\]
by the assumptions on $K_{BC}$. 
Hence, with (\ref{Dcontract}), from (\ref{halfprolong}) we get
\[
0= 3\D^A( X_{(A}F_{BC)})
=
\tfrac{(n+2)(n+6)}{n+4}
F_{BC}+ 2\D^A\left( X_{(B}F_{C)A}\right)
\]
From (\ref{DhatX}) and $X^AF_{BA}=0$,
\[
\D^A\left( X_{(B}F_{C)A}\right)
=
\tfrac{n+6}{n+4} F_{BC}
+
\tfrac{(n+4)^2-4}{(n+4)^2}
X_{(B}\D^AF_{C)A}
,
\]
so that
\begin{equation}
\label{Feq}
0
=
F_{BC}
+
\underbrace{\tfrac{2(n+4)^2-8}{(n+6)(n+4)}}_{=:c}
X_{(B}\D^AF_{C)A}.
\end{equation}
Since $c\not=0$, contracting with $X^B$ shows that $X^B\D^AF_{AB}=0$ and therefore, by applying $\D_C$ to this and using (\ref{DhatX}),
\[X^B\D_C\D^AF_{AB}
=
\tfrac{2}{n+2}X_C\D^A\D^BF_{AB}-\D^AF_{AB}.\]
Using this, (\ref{DhatX}) and (\ref{Dcontract}), when contracting $\D^B$ into (\ref{Feq}) we get
\[
0
=
b\, \D^BF_{BC}
+a\, X_C\D^A\D^BF_{AB},
\]
with some constants $a$ and  $b\not=0$. 
Hence, we have that $\D^BF_{BC}= \phi X_C$ for some function $\phi$. Therefore, (\ref{Feq}) yields that 
$F_{BC}= f X_BX_C$ with some function $f$, so that 
\[\D_{(A}K_{BC)}=fX_A X_BX_C.\]
When contracting this with $\D^A$, formula (\ref{Dcontract}) shows that
$0=fX_BX_C$, so that $F_{BC}=0$.
\end{proof}

\begin{remark} The conformally flat case can also be argued by using the classification of conformally invariant operators
  between weighted tensor bundles
  \cite{BoeCollingwood85-1,BoeCollingwood85-2,EastwoodSlovak97} -- as
  applied to the tractor slots. 
  \end{remark}

We remain in the conformally flat setting. For a conformal Killing
tensor we have the following assignment
\begin{equation}
\label{K4}
k_{ab}\stackrel{(\ref{rk2half})}{\longmapsto} K_{AB}\longmapsto\K_{ABCD}:=\D_A\D_C K_{BD}.\end{equation}
Since $K_{BD}$ is symmetric and since in the conformally flat case $[\D_A,\D_C]=0$, we have that $\K_{ABCD}$
is a section of the subbundle ${\mathcal K}$ of $\E_{ABCD}$ whose fibres are given by tensors with the following symmetries,
\begin{equation}
\label{Kspace}
\mathfrak{K}:=\{ X_{ABCD}\mid 
X_{ABCD}=X_{ADCB}= X_{CBAD}, X_{A(BCD)}=0,
X_{ABC}{}^C=0
\}.
\end{equation}
Tractors in $ {\mathcal K}$ enjoy the additional symmetry of pairs
$X_{ABCD}= X_{BADC}$.
Indeed,
\[
0=
X_{A(BCD)}
- X_{B(ADC)}
+
X_{C(BAD)}
-
X_{D(CBA)}
=
2 X_{ABCD}-2 X_{BADC}.
\]
Consequently, the symmetrisation over any three indices vanishes and is $X_{ABCD}$ is trace-free over any pair of indices.

We arrive at the correspondence between (\ref{theo1-1}) and  (\ref{theo1-3}) in Theorem \ref{theo1}.
\begin{corollary}
On a conformally flat manifold $M$, there is a one-to-one correspondence
between conformal Killing tensors and parallel tractors $\K_{ABCD}$ in
${\mathcal K}$.  Hence if $M$ is simply connected there is  on
a one-to-one correspondence
between conformal Killing tensors and elements of 
 the vector
space $\mathfrak K$.
\end{corollary}
\begin{proof}
For one direction of the correspondence, it remains to show that
$\K_{ABCD}$ defined in (\ref{K4}) is parallel for the conformal tractor
connection. We show that $\D_A\K_{BCDE}=0$. First we note that, since the $\D_A$'s commute, we have that 
$\D_A\K_{BCDE}$ is a section of $\E_{(ABD)(CE)}$. 
Secondly, the map
\[
\E_{(ABD)(CE)}\ni L_{ABCDE}\longmapsto L_{AB(DCE)}\in \E_{(AB)(DCE)}\]
is an isomorphism (see for example \cite{GoverLeistner19}). Therefore, with $\D_{A}\K_{B(CDE)}=0$, we have that $\D_A\K_{BCDE}=0$. Consequently, since $\K_{BCDE}$ is of weight zero, \[0=\D_A\K_{BCDE}\equiv Z_A{}^a\nabla_a\K_{BCDE}\mod X_A,Y_A,\] and therefore $\K_{BCDE}$ is parallel for the tractor connection. 

The reverse direction in this correspondence is given by the composition of the maps
\[
K_{ABCD}\longmapsto K_{BD}=\tfrac{1}{2} X^AX^C\K_{ABCD},\quad K_{AB} \longmapsto k_{ab} :=Z^A{}_aZ^B{}_b K_{AB}.\]
We have to check that $K_{AB}$ has the required properties and that $k_{ab} $ is conformally Killing. Since $\K_{(ABC)D}=0$, we have that  $X^AK_{AB}=0$. Moreover,  $\D_E\K_{ABCD}=0$, and from formula~(\ref{leibnizhat}) and $\D_EX^A=\delta_E{}^A$, so we have
\[
2\,\hD_EK_{BD}= \D_E(X^AX^C)\K_{ABCD}
=
X^C(\K_{EBCD} + \K_{EDCB})-\tfrac{2}{n+2}X_E \K_{AB}{}^A{}_D.
\]
Since $\K_{ABCD}$ is trace-free, this implies
\[
\hD_{(E}K_{BD)} =X^C\K_{C(BED)}=0,\]
which finishes the proof.
\end{proof}

Finally, the correspondence in Theorem~\ref{theo1} arises from the well-know isomorphism between the vector space $\mathfrak K$ as defined in (\ref{Kspace}) and the space of Weyl tensors
\begin{equation}
\label{Wspace}
\mathfrak W:=\{ X_{ABCD}\mid 
X_{ABCD}=-X_{BACD}= -X_{ABDC}, X_{A[BCD]}=0,
X_{ABC}{}^B=0
\},
\end{equation}
which is given in (\ref{K4toW4}). It is straightforward to check that its inverse is 
\begin{equation}
\label{WKinv}\W_{ABCD}\longmapsto L_{ABCD}=\tfrac{2}{3}\left( \W_{ABCD}+ \W_{ADCB}\right).\end{equation}

\subsection{Killing scales for conformal Killing tensors}

In the following we use the notion of a scale as defined in Section \ref{scalesec}.
\begin{definition}\label{sksdef}
Let ${k}_{ab}\in\Gamma( \E_{(ab)_0}[4])$ be a conformal Killing tensor and $\sigma \in\Gamma( \mathcal E[1])$ a scale. 
\begin{enumerate}
\item We say that  $\sigma $ is a
 {\em strong Killing scale (SKS) for $k_{ab}$} if $k_{ab}$ is a Killing tensor for the metric  $g_{ab}=\sigma^{-2}\mathbf{g}_{ab}$, and
 \item  that $\sigma$ is a {\em Killing scale (KS) for $k_{ab}$} if there is a Killing tensor for the metric $g_{ab}=\sigma^{-2}\mathbf{g}_{ab}$ whose trace-free part is equal to $k_{ab}$.
 \end{enumerate}
\end{definition}
First we consider SKS to get a tractorial characterisation of Lemma \ref{kdvlemma}.
\begin{proposition}\label{SKSprop-tractor}
Let $k_{ab}$ be a conformal Killing tensor and $K_{AB}$ its half prolongation as in (\ref{rk2half}), $\sigma$ a scale and $I_A=\D_A\sigma$ its scale tractor.
\begin{enumerate}
\item The scale $\sigma$ is a SKS if and only if $I^AK_{AB}=X_B F$ for some $F\in \mathcal E[1]$.
\item The scale  $\sigma$ is an Einstein SKS if and only if $I^AK_{AB}=0$ and $I_A$ is parallel for the conformal tractor connection.
\end{enumerate}
\end{proposition}
\begin{proof}
For $K_{AB}$ as in (\ref{rk2half}),
\[
K_{AB}=
k_{ab}{Z}_A{}^a{Z}_B{}^b 
- \rho_a {Z}_{(A}{}^a X_{B)}
+\rho X_AX_B
=\left(\begin{array}{ccc}  
  0 &0 & 0\\
0& k_{ab} &    -\tfrac{1}{2}\rho_a
\\
0& -\tfrac{1}{2}\rho^a&
\rho
\end{array}\right),
\]
and scale tractor $I_A=\D_A \si =
\sigma Y_A+\nabla_a\sigma Z_A{}^a -\frac{1}{n}(\Delta \si + J\si)X_A$, we have that
\begin{equation}
\label{IK}I^AK_{AB}=
\left( \rho\sigma-\tfrac{1}{2}\rho^a\nabla_a\sigma\right) X_B+\left(k_{ab} \nabla^a\sigma-\tfrac{\sigma}{2}\rho_b\right)Z_B{}^b.\end{equation}
Now assume that  $\sigma$ is a SKS. Then $k_{ab}$ is Killing for $g_{ab}=\sigma^{-2}\g_{ab}$ and hence $\rho_a=0$. Moreover, splitting the tractor bundle with respect to the metric $g_{ab}$, so that $\nabla_a\sigma=0$, we get that 
$I^AK_{AB}=\rho \sigma X_B$. 

Moreover, if $\sigma$ is an Einstein SKS, then $\ro_{ab}=cg_{ab}$ and, because $k_{ab}$ is trace-free, we have $\rho=0$ and so $I^AK_{AB}=0$.

Conversely, if $I^AK_{AB}= X_B F$, we must have  $k_{ab} \nabla^a\sigma-\tfrac{\sigma}{2}\rho_b=0$. Then with $\nabla_a\sigma=-\sigma \Upsilon_a$, from the formula (\ref{ro}) for the conformal change of $\rho_a$ we get that $k_{ab}$ is Killing for the metric $\sigma^{-2}\g_{ab}$.
If we assume in addition that $I_A$ is parallel, then $\sigma^{-2}\g_{ab}$ is Einstein.
\end{proof}
This proposition also illustrates the statement of Lemma \ref{kdvlemma}: if there are two Killing scales we can write the splitting of the tractor bundle with respect to one, for which $k_{ab}$ is also divergence free, and then equation (\ref{IK}) shows that for  the scale tractor $I_A$ of the other scale $\sigma$ we have that $I^AK_{AB}\equiv0\mod X_B $ if and only if $k_{ab}\nabla^a\sigma=0$, which is equation~(\ref{kdv}) in Lemma~\ref{kdvlemma}.

This has the following consequence for the full prolongation.

\begin{theorem}\label{SKSlemma-tractor}
Let $(M,g)$ be conformally flat and 
 $k_{ab}$ be a conformal Killing tensor with 
  full prolongation $\K_{ABCD}=\D_A\D_CK_{BD}$. 
 Then a scale $\sigma$ 
 is an Einstein SKS  if and only if 
 $I_A=\D_A\sigma$ is parallel and satisfies $I^C\K_{ABCD}=0$.
\end{theorem}
\begin{proof}
If we assume that $I^C\K_{ABCD}=0$, we  have
$X^AX^BI^C\K_{ABCD}=0$, and hence $I^CK_{CD}=0$. Since $I^A$ is parallel, Proposition \ref{SKSprop-tractor} yields that $\sigma$ is an Einstein SKS.

Conversely, if $\sigma $ is an Einstein scale then $I_A$ is parallel
and hence commutes with $\D_B$. Since $\si$ is a SKS, we get that
$I^AK_{AB}=0$ from Proposition \ref{SKSprop-tractor}. Applying
Thomas-$\D$ operators, this implies that $I^C\K_{ABCD}=0$.
\end{proof}

We define the {\em nullity} (at $p\in M$) of a tractor $L_{ABCD}$, with
Weyl-tensor-type symmetries, as the subspace $\{V^A\in \E^A|_p\mid
V^AL_{ABCD}=0\}$. If $L_{ABCD}$ is parallel, then the dimension of the
nullity is constant. In this case the nullity determines a
well-defined vector subbundle of the standard tractor bundle $\ce^A$.
We obtain the following consequences.

\begin{corollary}
On simply connected conformally flat spaces there exist conformal
Killing tensors $k_{ab}$ that do not admit Einstein SKS, i.e.~there
are no Einstein metrics in the conformal class for which $k_{ab}$ is
Killing.
\end{corollary}
\begin{proof} For $p\in M$ take a $\K_{ABCD}|_p$ with zero nullity.
  By parallel transport with respect to the flat conformal tractor
  connection we obtain a section $\K_{ABCD}$ of $\mathcal K$.
  Again because of flatness, the tractor connection is equal to the
  conformal Killing tensor prolongation connection. Hence, the
  section $\K_{ABCD}$ is the prolongation a conformal Killing tensor
  for which the prolongation has zero nullity.
\end{proof}

On the other hand by a similar argument to the proof of this Corollary
on simply connected conformally flat manifolds it is easy to construct
conformal Killing tensors whose prolongation has non-trivial nullity,
and hence for which there are Einstein SKS. Indeed even in Riemannian
signature there are examples for which the dimension of the
prolongation nullity is $n-1$.

\begin{corollary}
On conformally flat spaces, for a conformal Killing tensor $k_{ab}$,
the dimension of the space of Einstein SKS is bounded by the dimension
of the nullity of $\K_{ABCD}$, and on simply connected manifolds
is equal to this dimension.
\end{corollary}

Now we return to conformally curved manifolds and consider Killing scales. In particular we obtain the second part of Theorem \ref{theo0} from the
introduction.

\begin{proposition}\label{theo0b}
On a conformal manifold, let $k_{ab}$ be a conformal Killing tensor and $K_{AB}$ its half prolongation as in (\ref{rk2half}), $\sigma$ a scale and $I_A=\D_A\sigma$ its scale tractor.
Then 
 $\sigma$ is a KS if and only if there is a section $\lambda$ of $\E[2]$  such that
 \begin{equation}
 \label{K2KS}I^AK_{AB}= \tfrac{\sigma}{2} \D_B\lambda -  \lambda I_B+X_B F\end{equation} for some $F\in \mathcal E[1]$.
\end{proposition}
\begin{proof}
  We work in the scale $\sigma$, on the open set where $\si$ is non-vanishing,
  so that $I_A=\si Y_A-\frac{n}{2}\si \J X_A$ and
\[I^AK_{AB}=
\si \rho X_B-\tfrac{1}{2}\si \rho_bZ_B{}^b.\]
 On the other hand, for any $\lambda\in \Gamma(\E[2])$ we have that 
\[ \D_A \lambda = 
2 \lambda Y_A+
\nabla_a\lambda Z_A{}^a -\tfrac{1}{n+2}(\Delta \lambda + 2\J\lambda)X_A.\]
Hence, equation (\ref{K2KS}) is equivalent to
$
 \rho_b=-\nabla_a\lambda$,
 i.e.~to 
$ \rho_a= \tfrac{2}{n+2} \nabla^bk_{ab}=-\nabla_a \lambda$ being a gradient. By Lemma~\ref{divkclosed}, this is equivalent to $\sigma$ being an KS.
Moreover, as seen in Lemma~\ref{divkclosed}, in the metric $g_{ab}=\sigma^{-2}\g_{ab}$, the Killing tensor is 
$\bar{k}_{ab}=k_{ab}+\lambda \g_{ab}$.
\end{proof}

For Einstein scales, we will see the conditions on the full
prolongation and an improved result for the half prolongation in the
next section.

\subsection{Einstein scales and the full prolongation} \label{theo2-sec}
We have seen in Theorem \ref{theo1} that on conformally flat spaces
conformal Killing tensors are in one-to-one correspondence with
parallel tractors with Weyl-tensor symmetries.  Then we have observed
in Theorem \ref{SKSlemma-tractor} that Einstein SKS for conformal
Killing tensors $k_{ab}$ correspond to parallel scale tractors $I_A$
in the annihilator of $\K_{ABCD}$.  The following gives a related result for Killing tensors.  Here and throughout this
section when we refer to scales we mean strict scales in the sense
that for such a scale $\si$ is positive or negative everywhere. We use
the scale to trivialise the density bundles, and so weights will not be mentioned. 
\begin{theorem}\label{kthm}
On a  conformally flat manifold let $I_A$ be a parallel standard tractor such that $I_AI^A\neq 0$. 
There is a one-to-one correspondence between 
parallel tractors $R_{ABCD}$  that
have Riemann tensor symmetries and satisfy
$I^A R_{ABCD}=0$ and 
sections
$r_{ab}$ of $\E_{(ab)}$ that are Killing in the Einstein scale
$\si: =X^AI_A$. 
\end{theorem}
\begin{proof}
By \cite{GoverLeistner19} there is a one-to-one correspondence between Killing tensors and projective parallel tractors $R_{ABCD}$ for the
  projective class of the Einstein Levi-Civita connection.  As explained in Section \ref{projsec}, or \cite{GoverMacbeth14}, these projective parallel tractors
  correspond to conformal parallel tractors orthogonal to $I_A$. 
  \end{proof}

\begin{remark}
In fact a stronger result is available. On conformally curved
manifolds, there is a one-to-one correspondence between parallel
tractors $R_{ABCD}$ that have Riemann tensor symmetries and satisfy
$I^A R_{ABCD}=0$ and sections $r_{ab}$ of $\E_{(ab)}$ that are Killing
in the Einstein scale $\si: =X^AI_A$ and are normal solutions for the projective structure in the
sense of \cite{Leitner05,cgh11}. The proof is essentially the same as normal solutions, by definition, correspond to tractors that are parallel for the usual (i.e. normal) tractor connection. 
  \end{remark}

In the following, we  will use the following notation:  two symmetric tractors $Q_{AB}$ and $P_{AB}$ define a tractor $(Q\wedge P)_{ABCD}$ with Riemann symmetries  by
\[(Q\wedge P)_{ABCD}=Q_{AC}P_{BD}- Q_{BC}P_{AD}+Q_{BD}P_{AC}-Q_{AD}P_{BC}.\]

Furthermore, we denote the ``top slot'' of a conformal tractor
$R_{ABCD}$ that has the symmetries of the Riemann tensor by
\[\mathrm{top}(R_{ABCD})_{ac}:= 4\, Z^A{}_aZ^C{}_c X^BX^DR_{ABCD},\]
and  the trace-free part  of $R_{ABCD}$ is 
 \begin{equation}
 \label{tf}
 \mathring{R}_{ABCD}
 =
 R_{ABCD}- (g\wedge P)_{ABCD}
 =
 R_{ABCD}-g_{AC}P_{BD}+ g_{BC}P_{AD}-g_{BD}P_{AC}+g_{AD}P_{BC},\end{equation}
  where $P_{AC}=\tfrac{1}{n} \left(R_{ABC}{}^B-\J g_{AC}\right)$ and $\J=P_A{}^A=\tfrac{1}{2(n+1)}R_{AB}{}^{AB}$.
  Note that
  \[
 R_{ABCD}=  r_{ac}\left( Z_{[A}{}^a Y_{B]}Z_{[C}{}^c Y_{D]}\right) + \, \mbox{lower filtration components} ,\]
  if $r_{ac}=\mathrm{top} (R_{ABCD})_{ac}$.
 Moreover, given a scale tractor $I_A$, we denote by $I^\perp$ the tractors $T_{AB\ldots}$ such that $I^AT_{AB\ldots}=0$.
We will need the following observation.
\begin{lemma}\label{Plem} 
Let $ R_{ABCD}=\mathring{R}_{ABCD}+
 (g\wedge P)_{ABCD}
$ be a tractor with Riemann tensor symmetries and with trace-free part $\mathring{R}_{ABCD}$. If 
 $I^A$ is a standard tractor with $I^AR_{ABCD}=0$, then 
% \begin{equation}\label{useful}
$I^BP_{BD}=-\tfrac{\J}{n}\,I_D$,
where $\J=P_A{}^A$.
\end{lemma}

\begin{proof}
From (\ref{tf}) we get
$$
g^{AC}R_{ABCD}=n\, P_{BD}+g_{BD} \J.
$$
With $I^AR_{ABCD}=0$, this yields 
$
I^BP_{BD}=-\tfrac{\J}{n}\,I_D
$.
\end{proof}

In the following we want to analyse the situation of when an Einstein
scale is a Killing scale for a given conformal Killing tensor. We will
need some preliminary results.
\begin{lemma}\label{key0}
  If $r_{ab}$ is the top slot of a tractor $R_{ABCD}$ with Riemann tensor symmetries and $k_{ab}$ the top slot of the trace-free part $\mathring{R}_{ABCD}$ of $R_{ABCD}$, then $k_{ab}$ is  the trace-free part of $r_{ab}$, i.e. $\mathrm{top}$ and $\mathring{\ }$ commute,
  \[ \mathrm{top}(\mathring{R}_{ABCD})_{ab}=  \left(\mathrm{top} ({R}_{ABCD})\right)_{(ab)_0}.\]
\end{lemma}
\begin{proof} 
Contracting equation~(\ref{tf}) for the trace-free part of $R_{ABCD}$ 
with $X^B$ twice we get
$$
X^BX^DR_{ABCD}= X^BX^D \mathring{R}_{ABCD} +g_{AC}X^BX^DP_{BD} -2X^DX_{(A} P_{C)D}.
$$
Then we contract with $Z^A{}_a$ and $Z^C{}_c$, and use  $Z^A{}_aX_A=0$ to get  that 
$$
r_{ac} =k_{ac}+ 4 X^BX^DP_{BD} \bg_{ac},
$$
and so the tensor $k_{ac}$ is a trace adjustment of
$r_{ac}$. But $k_{ac}$ is trace-free as it is  the top slot of a trace-free
tractor with Weyl symmetries. Indeed,
\[k_a{}^a= 4\, \g_{ac} Z^{Aa}Z^{Cc}  X^BX^D\mathring{R}_{ABCD}=4
\left( g^{AC} -2 X^{(A}Y^{C)} \right) X^BX^D\mathring{R}_{ABCD}
=
0,
\]
because of (\ref{gAB}) and $\mathring{R}_{ABCD}$ being trace-free and with Riemann symmetries. 
Hence, $k_{ab}$ is the trace-free part of $r_{ab}$.
\end{proof}

Given Theorem \ref{kthm} here, one might wonder if the trace-free
part of $r_{ab}$ --- the conformal Killing part --- has conformal prolongation
orthogonal to $I_A$. In general this is not the case, as the following proposition shows.
\begin{proposition}\label{key1}
On a  conformally flat manifold,
let $k_{ab}$ be a conformal Killing tensor with full conformal prolongation $\W_{ABCD}$, with Weyl symmetries. Assume that there is a function $\lambda$ such that
$r_{ab}=k_{ab}+\lambda g_{ab}$ is a Killing tensor in the Einstein scale $\si
=X^AI_A$, with $I^AI_A\neq 0$, where $g_{ab}=\sigma^{-2}\g_{ab}$ is the Einstein metric corresponding to the scale $\sigma$.
 Then $I^A\W_{ABCD}=0$ if and only if
$\lambda$  is constant (i.e. if also $k_{ab}$ is  Killing for the metric $g_{ab}$).
\end{proposition}
\begin{proof}
First we prove the `if' direction, so assume that $r_{ab}=k_{ab}+\lambda g_{ab}$, with $\lambda $ constant and $k_{ab}$ the trace-free part of
$r_{ab}$, is a Killing tensor for the Einstein metric
$g_{ab}=\sigma^{-2}\g_{ab}$. In the following we use this scale to trivialise the density bundles. Since $\lambda$ is constant,
$k_{ab}=r_{ab}-\lambda g_{ab}$ is a trace-free Killing tensor for
$g_{ab}$. This means that $\sigma$ is an Einstein SKS so that we can apply
Theorem~\ref{SKSlemma-tractor} to get that the conformal prolongation
of $k_{ab}$ is orthogonal to the scale tractor $I_A$.

For the converse assume that $R_{ABCD}$ is the tractor that corresponds via Theorem~\ref{kthm}  to the Killing tensor $r_{ab}=k_{ab}+\lambda g_{ab}$, that is, $R_{ABCD}$ has Riemann symmetries, is parallel and satisfies $I^AR_{ABCD}=0$. 
  Then we have equation (\ref{tf}) with $ \mathring{R}_{ABCD}$ being the trace-free part of ${R}_{ABCD}$. As observed in the proof of Lemma \ref{key0}, the top slot of $\mathring{R}_{ABCD}$ is the conformal Killing tensor $k_{ab}$, so when 
 taking the top slots of (\ref{tf}), i.e.~contracting (\ref{tf}) with 
 $Z^{A}{}_{a} X^B Z^{C}{}_{c}X^D$, we get
 \[
r_{ac} =k_{ac}+ 4 X^BX^DP_{BD} g_{ac},\]
as in the proof of Lemma \ref{key0}, so it remains to show that $X^BX^DP_{BD} $ is a constant.
For this we use that $R_{ABCD}$ is  parallel, and so is its trace-free part $\mathring{R}_{ABCD}$. By Lemma \ref{key0} and Theorem \ref{theo1} this implies 
  that $\mathring{R}_{ABCD}=\W_{ABCD}$ is the prolongation of the conformal Killing tensor $k_{ab}$. 
  From this and  from (\ref{tf}) we get
  $$
0= I^A\W_{ABCD}+ I_C P_{BD}- g_{BC} I^A P_{AD}+g_{BD}I^AP_{AC}-I_D P_{BC} ,
  $$
 where $ \W_{ABCD}$ is the prolongation of $k_{ab}$.  
From Lemma~\ref{Plem} we get 
$I^BP_{BD}=-\tfrac{\J}{n}\,I_D$, where $\J=g^{BD} P_{BD}$ is constant,
and hence 
$$
0= I^A\W_{ABCD}+ I_C P_{BD}+\tfrac{\J}{n} g_{BC} I_D 
-\tfrac{\J}{n}g_{BD}I_C  -I_D P_{BC}  .
$$ 
Now we use the assumption on the prolongation of $k_{ab}$  that  $I^A\W_{ABCD}=0$. Contracting the last display with $I^C$ gives
%\begin{equation}\label{step}
\[
0=  \iota P_{BD}+\tfrac{2\J}{n} I_B I_D 
-\tfrac{\J\, \iota}{n}g_{BD}  
%+ \tfrac{\J}{n}I_D I_B 
  ,
\]
where $\iota:=g_{AB}I^AI^B$ constant. Therefore,
\[
X^BX^DP_{BD} =
- \tfrac{2\J}{\iota n}X^BI_BX^DI_D=- \tfrac{2\J}{\iota n}
\]
is a constant, where we use that $X^BI_BX^DI_D=1$ 
in the Einstein scale. 
\end{proof}

 The key step in analysing when a conformal Killing
tensor is, after possible trace adjustment, Killing for an Einstein
metric is the following result.
\begin{lemma}\label{corol1}
On a conformally flat manifold, let $k_{ab}$ be a conformal Killing
tensor with full conformal prolongation $\W_{ABCD}$ and let $\si $ be
an Einstein scale, with corresponding scale tractor $I_A$, that
satisfies $I_AI^A\neq 0$. Then $\sigma$ is a Killing scale,
i.e.~$k_{ab}$ is, after trace adjustment, Killing for the Einstein
scale $\si$, if and only if $\W_{ABCD}$ is the trace-free part of some
parallel $R_{ABCD}$ that has Riemann symmetries and satisfies $I^A
R_{ABCD}=0$.
\end{lemma}
\begin{proof}
Let $k_{ab}$ be conformally Killing and $\W_{ABCD}$ its full prolongation.

First we prove the `only if' direction.  Assume that $r_{ab}=k_{ab}+
f\, g_{ab}$ is Killing for an Einstein metric
$g_{ab}=\sigma^{-2}\g_{ab}$ and a function $f$. Then by Theorem
\ref{SKSlemma-tractor}, $r_{ab}$ defines a parallel tractor $R_{ABCD}$
with Riemann tensor symmetries, such that $I^AR_{ABCD}=0$. We denote
trace-free part by $\mathring{R}_{ABCD}$.  By Lemma \ref{key0}, the top
slot of $\mathring{R}_{ABCD}$ is $k_{ab}$, the trace-free part of
$r_{ab}$, and 
because $\mathring{R}_{ABCD} $ and $\W_{ABCD}$ are
parallel with the same top slot 
(and are sections of a bundle induced by an irreducible representation of $\mathrm{SO}(r+1,s+1)$), 
they must be equal.
It is easily checked using~(\ref{nabXYZ}) that there are no non-trivial parallel tractors (with these symmetries) that have zero top slot.

 For the converse, assume that there is an $R_{ABCD}$ with Riemann symmetries such that 
 $I^A R_{ABCD}=0$ 
 and with trace-free part $\W_{ABCD}$, the full prolongation of $k_{ab}$. By Theorem~\ref{kthm}, the top slot $r_{ab}$ of $R_{ABCD}$ is a Killing tensor for a metric $g_{ab}=\sigma^{-2}\g_{ab}$, and by Lemma~\ref{key0}, the conformal Killing tensor $k_{ab}$ is the trace-free part of the Killing tensor $r_{ab}$. Hence, $r_{ab} =k_{ab} +\tfrac{r_a{}^a}{n} g_{ab} $ as required.
 \end{proof}

This means that we need to characterise tractors with algebraic Weyl symmetries that arise as in Lemma \ref{corol1}, i.e.~as trace-free part of a tractor $R_{ABCD}$ that has Riemann symmetries and is orthogonal to $I_A$. The following Lemma shows that such  a $R_{ABCD}$ is not uniquely determined  by its the trace-free part.

\begin{lemma} \label{tflemma}Let $\phi$ be the endomorphism on the space of curvature tensors that assigns to $R_{ABCD}$ its trace-free part $\mathring{R}_{ABCD}$ as in (\ref{tf}), and for a scale tractor $I^A$ let $I^\perp$ denote the curvature tensors with $I^AR_{ABCD}=0$.
If $I^A$ satisfies $\iota=I_AI^A\not=0$, then
$\ker(\phi)\cap I^\perp$ is spanned by \[
\left( g\wedge \left( g-\tfrac{2}{\iota}I^2\right)\right)_{ABCD},\]
where $I^2$ denotes the symmetric tractor $I_AI_B$.

\begin{proof}
For $R_{ABCD}:=\left(g\wedge \left( g-\tfrac{2}{\iota}I^2\right)\right)_{ABCD}$ it is straightforward to check that $I^AR_{ABCD}=0$. Moreover,
we have that 
\[R_{ABC}{}^B= 2n\left(g_{AC}-\tfrac{1}{\iota}I_AI_C\right)
%=2n \Pi g_{AC}
,\]
and hence that $R_{AB}{}^{AB}=2n(n+1)$. 
With this, a direct computation shows that  $P_{AC}= g_{AC}-\tfrac{2}{\iota} I_AI_C$, which implies that $\phi(R_{ABCD})=0$.

Conversely, let  $R_{ABCD}=(g\wedge   P)_{ABCD}$, with $P_{AB}$ symmetric, be in the kernel of $\phi$. If we assume in addition  that $I^AR_{ABCD}=0$, we also have that 
$I^AR_{AB}=0$ and hence
\[I^AP_{AB}=-\tfrac{\J}{n}I_B,\qquad I^AI^BP_{AB}=-\tfrac{\J}{n}\iota,\]
where $\J=P_A{}^A$. 
Using this, we get
\[0=I^AR_{ABCD}=2I_{[C}P_{D]B} -\tfrac{\J}{n} I_{[C}g_{D]B},\]
and by contracting this with  $I^C$ again,
\[0=\iota P_{BD}+\tfrac{\J}{n}\left( 2 I_BI_D- \iota g_{BD}\right).\]
This proves the result.
\end{proof}

\end{lemma}

We are now ready to prove the main theorem of this section.
\begin{theorem}\label{theo2a}
On a conformally flat manifold, let $k_{ab}$ be a conformal Killing
tensor with full conformal prolongation $\W_{ABCD}$, where $\W_{ABCD}$
has Weyl tensor symmetries. Let $\si$ be an Einstein scale, with the
corresponding scale tractor $I$ satisfying $\iota=I^AI_A\neq 0$.  Then $\si$ is a Killing scale for $k_{ab}$  if and only if
\begin{equation}\label{IWwedgeI} 0=I^A\W_{AB[CD}I_{E]}.\end{equation}
  \end{theorem}

\begin{proof}
Let $k_{ab}$ be conformally Killing and $\W_{ABCD}$ its full prolongation, which is parallel.

First we prove the `if' direction. 
Note that  the assumption
\[
 0
 =
 3 \, I^A\W_{AB[CD}I_{E]}
 =
 I^A\W_{ABCD}I_{E}+2\, I^A\W_{ABE[C}I_{D]}
 \]
is equivalent to
\begin{equation}
\label{IinK}
 0
 =
\iota  I^A\W_{ABCD}+2\, I^AI^E \W_{ABE[C}I_{D]}.
 \end{equation}
Since $\W_{ABCD}$ and $I^A$ are parallel, \begin{equation}\label{Pdef}
R_{ABCD}:=
\iota 
\W_{ABCD}+(g\wedge P)_{ABCD},\qquad\text{ 
with }\quad
P_{BD}:=
- I^AI^C\W_{ABCD},\end{equation}
 is parallel and by (\ref{IinK}) orthogonal to $I^A$.
 Indeed, since 
 $P_{B[C}I_{D]} = 
 -
  I^AI^E\W_{ABE[C}I_{D]}$
  and
 $I^AP_{AC}=- I^AI^B I^D\W_{BADC}=0$,
we get by (\ref{IinK}) that
\begin{eqnarray*}
I^AR_{ABCD}&=&
\iota I^A\W_{ABCD}-2P_{B[C}I_{D]} -2g_{B[C}P_{D]A}I^A
\\
&=&
\iota I^A\W_{ABCD}+2 I^AI^E\W_{ABE[C}I_{D]}
\\
&=&0.
\end{eqnarray*}
 Hence, with $R_{ABCD}$ parallel and $I^AR_{ABCD}=0$, 
 we can apply the  Lemma~\ref{corol1}, to get that a trace adjustment $r_{ab}$ of $k_{ab}$ is Killing for the scale corresponding to $I^A$.

 Note that in the definition of $P_{BD}$ in (\ref{Pdef}), we could have added a constant multiple of $ \iota g_{BD}   -2 I_B I_D$ to $P_{BD}$ without changing the argument. This is due to Lemma \ref{tflemma}.

Conversely, assume that  for the scale corresponding to $I_A$ there is Killing tensor $r_{ab}$  with trace-free part $k_{ab}$. By Lemma~\ref{corol1}, there is a parallel $R_{ABCD}$ with Riemann tensor symmetries and orthogonal to $I_A$ such that $\W_{ABCD}$ is the trace-free part of $R_{ABCD}$. 
  From $I^AR_{ABCD}=0$, and  with (\ref{tf}) and using Lemma~\ref{Plem}, we get
  \begin{eqnarray*}
0&=& I^A\W_{ABCD}+ I_C P_{BD}- g_{BC} I^A P_{AD}+g_{BD}I^AP_{AC}-I_D P_{BC} ,\\
&=& I^A\W_{ABCD}+  \tfrac{\J}{n} g_{BC}  I_D -  \tfrac{\J}{n}g_{BD}  I_C + I_C P_{BD}-I_D P_{BC} .
\end{eqnarray*}
This implies 
$
0= I^A\W_{AB[CD}I_{E]}$ as required.\end{proof}

The above theorem provides a proof of the correspondence between (\ref{theo2-1}) and  (\ref{theo2-3}) in Theorem~\ref{theo2}. 
This allows us to determine the dimension of the space of conformal Killing tensors that can be trace adjusted to become Killing for a given Einstein scale.
\begin{corollary}\label{dim-cor}
On a conformally flat manifold, let
$\si$ be an Einstein scale, with the
corresponding scale tractor $I$ satisfying $\iota=I^AI_A\neq 0$.
Then the space of conformal Killing tensors for which $\sigma$ is a Killing scale is isomorphic to the vector space 
\[ I_A^\perp/(\mathbb{R}\cdot \mathbb{I}_{ABCD}),\quad \text{ where }\quad 
\mathbb{I}_{ABCD}:=\left( g\wedge \left( g-\tfrac{2}{\iota}I^2\right)\right)_{ABCD}, \]
and
 $I_A^\perp$ is the space of curvature tensors $R_{ABCD}$ that satisfy $I^A R_{ABCD}=0$.
In particular, the dimension of this space is
equal to $\frac{n(n+1)^2(n+2)}{12}-1$.
\end{corollary}
\begin{proof}
  We denote by $\mathfrak{R}^N$ be the space of tensors with Riemann
  type symmetries on an $N$-dimensional pseudo-Euclidean inner product
  space, and by $\mathfrak{W}^N$ be the space of Weyl tensors, i.e. of
  trace-free curvature tensors. Let $\phi:\mathfrak{R}^{n+2}\to
  \mathfrak{W}^{n+2}$ be the map in~(\ref{tf}) that sends an
  $R_{ABCD}$ to its trace-free part.  Let $\mathfrak{C}$ be the
  subspace of $\mathfrak{W}^{n+2} $ consisiting of elements that satisfy
  equation~(\ref{IWwedgeI}), or equivalently equation~(\ref{IinK}), and
  $\psi: \mathfrak{C}\to I_A^\perp \subset \mathfrak{R}^{n+2}$ be the
  map in~(\ref{Pdef}).  Then we have the diagram
\[\begin{array}{ccccc}
\E_{(AB)}|_p &\hookrightarrow & \mathfrak{R}^{n+2} & \stackrel{\phi}{\twoheadrightarrow}& \mathfrak{W}^{n+2}
\\
\cup&&\cup&&\cup
\\
\mathbb{R}\cdot \left( g_{AB}-\tfrac{2}{\iota}I_AI_B\right)|_p&\hookrightarrow & I_A^\perp &\stackrel{\psi}{\longleftarrow}&\mathfrak{C}.
\end{array}\]
A direct computation verifies that   $P_{AB}$ from~(\ref{Pdef}) satisfies $\phi((g\wedge P)_{ABCD})=0$, so that 
$\phi\circ \psi$ is the identity on $\mathfrak{C}$. Consequently,  $\phi|_{I_A^\perp}$ is surjective onto $\mathfrak{C}$. Moreover, in Lemma~\ref{tflemma} we have seen that the kernel of $\phi|_{I_A^\perp}$ is spanned by $\mathbb{I}_{ABCD}:=\left( g\wedge \left( g-\tfrac{2}{\iota}I^2\right)\right)_{ABCD} $.
Hence, $\mathfrak{C}$ is isomorphic to $I_A^\perp/\mathbb{R}\cdot \mathbb{I}_{ABCD}$, and so the statement follows from Theorem~\ref{theo2a}. For the  computation of the dimension, note that $I^\perp_A\subset \mathfrak{R}^{n+2}$ is isomorphic to $\mathfrak{R}^{n+1}$. 
\end{proof}

In order to obtain an improvement of the result for the half prolongation in  Proposition~\ref{theo0b}, we use the isomorphism between the space of Weyl tensors and the space $\mathfrak K$ as defined in (\ref{Kspace}).

Writing out the isomorphism between $\K_{ABCD} $ and $\W_{ABCD}$ yields the equivalence between (\ref{theo2-1})  in Theorem \ref{theo2} and an equation on the full prolongation $\K_{ABCD}$.
\begin{corollary}
On a  conformally flat manifold,
let $k_{ab}$ be a conformal Killing tensor with full conformal prolongation $\K_{ABCD}=\D_A\D_CK_{BD}$, where $\K_{ABCD}$ is a parallel section of ${\mathcal K} \subset \E_{ABCD}$.  Let $\si$ be an Einstein scale, with the
corresponding scale tractor $I$ satisfying $\iota=I^AI_A\neq 0$.
Then $\si $ is a Killing scale  if and only if 
\begin{equation}\label{IKwedgeI} 
\iota  I^A\K_{ABCD}= I^AI^E \K_{ABE[D}I_{C]} + I^AI^E \K_{ADE[B}I_{C]}.
\end{equation}
\end{corollary}
\begin{proof} The relation between the full prolongation with Weyl tensor symmetries and $\K_{ABCD}=\D_A\D_CK_{BD}$ is given by equation (\ref{WKinv}).
Hence, equation (\ref{IWwedgeI}), or equivalently (\ref{IinK}), yields
\[
\iota I^A\K_{ABCD}=
\tfrac{2\iota}{3} I^A\left(\W_{ABCD}+\W_{ADCB}\right)
=
\tfrac{4\iota}{3} I^AI^E \left(\W_{ABE[D}I_{C]}+\W_{ADE[B}I_{C]}\right)
\]
With $W_{ABCD}=\K_{[AB][CD]}$, we get 
 \[ I^AI^E \W_{ABEC}=\tfrac{1}{2}I^AI^E\left(\K_{ABEC}- \K_{AEBC}\right).\]
 For the second term we have, using $K_{(AEB)C}=0$, that
 \[I^AI^E \K_{AEBC}=\tfrac{1}{2} I^AI^E \left( \K_{AEBC}+ \K_{EABC}\right)
 =
 \tfrac{1}{2} I^AI^E \left( \K_{AEBC}+ \K_{BAEC}\right)=- \tfrac{1}{2} I^AI^E \K_{EBAC},\]
 so that 
 $ I^AI^E \W_{ABEC}=\frac{3}{4}I^AI^E\K_{ABEC}$. This implies equation (\ref{IKwedgeI}).
\end{proof}
Now  we can use that \[2K_{BD}=X^AX^C\K_{ABCD}= X^AX^C\K_{BADC},\]
because of the pairwise symmetry $\K_{ABCD}=\K_{BADC}$, see Section~\ref{theo1-sec},  and 
\[ \D_CK_{BD}=X^A\K_{ABCD}= X^A\K_{BADC},\] to get an improvement of 
Proposition \ref{theo0b}. The following corollary establishes the equivalence between (\ref{theo2-1}) and  (\ref{theo2-2}) in Theorem~\ref{theo2}
\begin{corollary}\label{theo2b}
On a conformally flat manifold, let $k_{ab}$ be a conformal Killing tensor and $K_{AB}$ its half prolongation as in (\ref{rk2half}), $\sigma$ an Einstein scale scale and $I_A=\D_A\sigma$ its parallel scale tractor with $\iota=I_AI^A\not=0$.
Then 
 $\sigma$ is a Killing scale if and only if there is a section $\kappa$ of $\E[2]$  such that
 \begin{equation}
 \label{K2EKS} I^AK_{AC}=\tfrac{\sigma}{2} \D_B\kappa-\kappa I_B \end{equation} for some $F\in \mathcal E[1]$.
\end{corollary}
\begin{proof}
We contract equation (\ref{IKwedgeI}) with $X^B X^D$ and use the pairwise symmetry $\K_{ABCD}=\K_{BADC}$ and $X^BI_B=\sigma$ to get
\[
2\iota I^AK_{AC}=
I^AI^E\left(2K_{AE} I_C- \sigma\D_CK_{AE}\right).
\]
Now we define $\kappa:=-\frac{1}{\iota} I^AI^EK_{AE}$.
Since $I_A$ is parallel, it commutes with $\D_C$, so we get that
\[
2 I^AK_{AC}=
\sigma \D_C\kappa- 2\kappa I_C,
\]
completing the proof.
\end{proof}

In the last part of this section we  obtain an analogue of Proposition~\ref{newkillingvec}
using 
   the relation to the  projective setting in Section \ref{projsec}.
For this let   
  $\Pi_A{}^B$ be the projection from $\E_{A}$ to $I^\perp$, which  induces a projection from  the tractors with Weyl symmetries, to $I^\perp$  in $\E_{ABCD}$ given by
  \[\bar{W}_{ABCD}=\Pi_A{}^E\Pi_B{}^F\Pi_C{}^H\Pi_D{}^GW_{EFGH}.\]
  Note that $\bar{W}_{ABCD}$ still has Riemannian symmetries but is no longer trace-free. 
The  analogue of Proposition~\ref{newkillingvec} is as follows:
\begin{proposition}\label{newkillingtensor}
    On a conformally flat manifold, let $k_{ab}$ be a conformal
    Killing tensor with full conformal prolongation $\W_{ABCD}$, where
    $\W_{ABCD}$ has Weyl-tensor-type symmetries. If $\sigma$ is an
    Einstein scale with scale tractor $I_A$, then the projection
    $\bar{W}_{ABCD}$ of ${W}_{ABCD}$ onto $I^\perp$ is parallel. In
    particular, it defines a Killing tensor $\bar{k}_{ab}$ for the
    metric $\sigma^{-2}\g_{ab}$.%%,
  \end{proposition}
\begin{proof}
  Since the scale tractor $I_A$ is parallel for the conformal connection, then the projection is parallel. Therefore, if $W_{ABCD}$ is parallel, then so is $\bar{W}_{ABCD}$. As explained in Section \ref{projsec}, see also \cite{GoverMacbeth14},  $\bar W_{ABCD}$ is parallel for the projective tractor connection and hence defines a Killing tensor $\bar{k}_{ab}$ for the metric $g_{ab}=\sigma^{-2}\g_{ab}$.
\end{proof}
 
%\bibliographystyle{abbrv}

%\bibliography{geobib}

\begin{thebibliography}{10}

\bibitem{AnderssonBlue15}
L.~Andersson and P.~Blue.
\newblock Hidden symmetries and decay for the wave equation on the {K}err
  spacetime.
\newblock {\em Ann. of Math. (2)}, 182(3):787--853, 2015.

\bibitem{bailey-eastwood-gover94}
T.~N. Bailey, M.~G. Eastwood, and A.~R. Gover.
\newblock Thomas's structure bundle for conformal, projective and related
  structures.
\newblock {\em Rocky Mountain J. Math.}, 24(4):1191--1217, 1994.

\bibitem{Benn06}
I.~M. Benn.
\newblock Geodesics and {K}illing tensors in mechanics.
\newblock {\em J. Math. Phys.}, 47(2):022903, 15, 2006.

\bibitem{BoeCollingwood85-1}
B.~D. Boe and D.~H. Collingwood.
\newblock A comparison theory for the structure of induced representations.
\newblock {\em J. Algebra}, 94(2):511--545, 1985.

\bibitem{BoeCollingwood85-2}
B.~D. Boe and D.~H. Collingwood.
\newblock A comparison theory for the structure of induced representations.
  {II}.
\newblock {\em Math. Z.}, 190(1):1--11, 1985.

\bibitem{BoyerKalninsMiller86}
C.~P. Boyer, E.~G. Kalnins, and W.~Miller, Jr.
\newblock St{\"a}ckel-equivalent integrable hamiltonian systems.
\newblock {\em SIAM Journal on Mathematical Analysis}, 17(4):778--797, 1986.

\bibitem{BransonCapEastwoodGover06}
T.~Branson, A.~{\v{C}}ap, M.~Eastwood, and A.~R. Gover.
\newblock Prolongations of geometric overdetermined systems.
\newblock {\em Internat. J. Math.}, 17(6):641--664, 2006.

\bibitem{BryantChernGardnerGoldschmidtGriffiths91}
R.~L. Bryant, S.~S. Chern, R.~B. Gardner, H.~L. Goldschmidt, and P.~A.
  Griffiths.
\newblock {\em Exterior differential systems}, volume~18 of {\em Mathematical
  Sciences Research Institute Publications}.
\newblock Springer-Verlag, New York, 1991.

\bibitem{CalderbankDiemer01}
D.~M.~J. Calderbank and T.~Diemer.
\newblock Differential invariants and curved {B}ernstein-{G}elfand-{G}elfand
  sequences.
\newblock {\em J. Reine Angew. Math.}, 537:67--103, 2001.

\bibitem{cgh11}
A.~{\v{C}}ap, A.~R. Gover, and M.~Hammerl.
\newblock Holonomy reductions of {C}artan geometries and curved orbit
  decompositions.
\newblock {\em Duke Math. J.}, 163(5):1035--1070, 2014.

\bibitem{CapSlovakSoucek01}
A.~{\v{C}}ap, J.~Slov{\'a}k, and V.~Sou{\v{c}}ek.
\newblock Bernstein-{G}elfand-{G}elfand sequences.
\newblock {\em Ann. of Math. (2)}, 154(1):97--113, 2001.

\bibitem{Cariglia14}
M.~Cariglia.
\newblock Hidden symmetries of dynamics in classical and quantum physics.
\newblock {\em Rev. Mod. Phys.}, 86:1283--1336, Dec 2014.

\bibitem{Carter68}
B.~Carter.
\newblock Global structure of the {K}err family of gravitational fields.
\newblock {\em Phys. Rev.}, 174:1559--1571, Oct 1968.



\bibitem{curry-Gover} S.N. Curry, A.R. Gover.
  \newblock An
  introduction to conformal geometry and tractor calculus, with a view
  to applications in general relativity
\newblock London Math. Soc. Lecture Note Ser., 443
Cambridge University Press, Cambridge, 2018, 86--170.
  
\bibitem{eastwood05}
M.~Eastwood.
\newblock Higher symmetries of the {L}aplacian.
\newblock {\em Ann. of Math. (2)}, 161(3):1645--1665, 2005.

\bibitem{EastwoodSlovak97}
M.~Eastwood and J.~Slov\'{a}k.
\newblock Semiholonomic {V}erma modules.
\newblock {\em J. Algebra}, 197(2):424--448, 1997.

\bibitem{Eisenhart28}
L.~P. Eisenhart.
\newblock Dynamical trajectories and geodesics.
\newblock {\em Ann. of Math. (2)}, 30(1-4):591--606, 1928/29.

\bibitem{Frances07}
C.~Frances.
\newblock Causal conformal vector fields, and singularities of twistor spinors.
\newblock {\em Ann. Global Anal. Geom.}, 32(3):277--295, 2007.

\bibitem{Gover06}
A.~R. Gover.
\newblock Laplacian operators and {$Q$}-curvature on conformally {E}instein
  manifolds.
\newblock {\em Math. Ann.}, 336(2):311--334, 2006.

\bibitem{Gover10}
A.~R. Gover.
\newblock Almost {E}instein and {P}oincar\'{e}-{E}instein manifolds in
  {R}iemannian signature.
\newblock {\em J. Geom. Phys.}, 60(2):182--204, 2010.

\bibitem{GoverLeistner19}
A.~R. Gover and T.~Leistner.
\newblock Invariant prolongation of the {K}illing tensor equation.
\newblock {\em Ann. Mat. Pura Appl. (4)}, 198(1):307--334, 2019.

\bibitem{GoverMacbeth14}
A.~R. Gover and H.~R. Macbeth.
\newblock Detecting {E}instein geodesics: {E}instein metrics in projective and
  conformal geometry.
\newblock {\em Differential Geom. Appl.}, 33(suppl.):44--69, 2014.

\bibitem{gover-nurowski04}
A.~R. Gover and P.~Nurowski.
\newblock Obstructions to conformally {E}instein metrics in {$n$} dimensions.
\newblock {\em J. Geom. Phys.}, 56(3):450--484, 2006.

\bibitem{Gover-Peterson-Lap} A.~R. Gover and L. Peterson. 
 \newblock Conformally invariant powers of the Laplacian, Q-curvature, and
 tractor calculus.
 \newblock {\em Comm. Math. Phys.}, 235 (2003), no. 2, 339--378.

 
  
\bibitem{GoverSnellTaghavi}
A.~R. Gover, D.~Snell, and A.~Taghavi-Chabert.
\newblock Distinguished curves and integrability in {R}iemannian, conformal,
  and projective geometry.
\newblock {\em Advances in Theoretical and Mathematical Physics}, 25(8):in
  press, 2022.

\bibitem{HammerlSombergSoucekSilhan12}
M.~Hammerl, P.~Somberg, V.~Sou{\v{c}}ek, and J.~{\v{S}}ilhan.
\newblock On a new normalization for tractor covariant derivatives.
\newblock {\em J. Eur. Math. Soc. (JEMS)}, 14(6):1859--1883, 2012.

\bibitem{KalninsKressMiller18}
E.~G. Kalnins, J.~M. Kress, and W.~Miller, Jr.
\newblock {\em Separation of variables and superintegrability}.
\newblock IOP Expanding Physics. IOP Publishing, Bristol, 2018.
\newblock The symmetry of solvable systems.

\bibitem{KalninsMiller82}
E.~G. Kalnins and W.~Miller, Jr.
\newblock Intrinsic characterisation of orthogonal {$R$} separation for
  {L}aplace equations.
\newblock {\em J. Phys. A}, 15(9):2699--2709, 1982.

\bibitem{Kostant55}
B.~Kostant.
\newblock Holonomy and the {L}ie algebra of infinitesimal motions of a
  {R}iemannian manifold.
\newblock {\em Trans. Amer. Math. Soc.}, 80:528--542, 1955.

\bibitem{KressSchobelVollmer24}
J.~Kress, K.~Sch\"{o}bel, and A.~Vollmer.
\newblock Algebraic conditions for conformal superintegrability in arbitrary
  dimension.
\newblock {\em Comm. Math. Phys.}, 405(4):Paper No. 92, 53, 2024.

\bibitem{KuhnelRademacher08}
W.~K{\"u}hnel and H.-B. Rademacher.
\newblock Conformal transformations of pseudo-{R}iemannian manifolds.
\newblock In {\em Recent developments in pseudo-{R}iemannian geometry}, ESI
  Lect. Math. Phys., pages 261--298. Eur. Math. Soc., Z\"urich, 2008.

\bibitem{Leitner05}
F.~Leitner.
\newblock Conformal {K}illing forms with normalisation condition.
\newblock {\em Rend. Circ. Mat. Palermo (2) Suppl.}, (75):279--292, 2005.

\bibitem{MichelRadouxSilhan14}
J.-P. Michel, F.~Radoux, and J.~{\v{S}}ilhan.
\newblock Second order symmetries of the conformal {L}aplacian.
\newblock {\em SIGMA Symmetry Integrability Geom. Methods Appl.}, 10:Paper 016,
  26, 2014.

\bibitem{MichelSombergv-Silhan17}
J.-P. Michel, P.~Somberg, and J.~\v~Silhan.
\newblock Prolongation of symmetric {K}illing tensors and commuting symmetries
  of the {L}aplace operator.
\newblock {\em Rocky Mountain J. Math.}, 47(2):587--619, 2017.

\bibitem{Miller77book}
W.~Miller, Jr.
\newblock {\em Symmetry and separation of variables}.
\newblock Addison-Wesley Publishing Co., Reading, Mass.-London-Amsterdam, 1977.
\newblock With a foreword by Richard Askey, Encyclopedia of Mathematics and its
  Applications, Vol. 4.

\bibitem{Smirnov06}
R.~G. Smirnov.
\newblock On the classical {B}ertrand-{D}arboux problem.
\newblock {\em Fundam. Prikl. Mat.}, 12(7):231--250, 2006.

\bibitem{Sommers73}
P.~Sommers.
\newblock On {K}illing tensors and constants of motion.
\newblock {\em J. Mathematical Phys.}, 14:787--790, 1973.

\end{thebibliography}

%\end{document}

\providecommand{\MR}[1]{}\def\cprime{$'$} \def\cprime{$'$} \def\cprime{$'$}

\end{document}